\documentclass[12pt]{article}

\usepackage[all]{xypic}
\usepackage{amsmath}
\usepackage{amsfonts}
\usepackage{amssymb}
\usepackage{amsthm}

\usepackage[top=2cm, bottom=2cm, left=2cm, right=2cm]{geometry}

\newtheorem{thm}{Theorem}[section]
\newtheorem{mthm}{Theorem}

\newtheorem{crl}[thm]{Corollary}
\newtheorem{mcrl}[mthm]{Corollary}

\newtheorem{prp}[thm]{Proposition}
\newtheorem{mprp}[mthm]{Proposition}

\newtheorem{lmm}[thm]{Lemma}
\newtheorem{mlmm}[mthm]{Lemma}

\newtheorem{mconj}[mthm]{Conjecture}

\newtheorem{rmk}[thm]{Remark}
\newtheorem{mrmk}[mthm]{Remark}

\newtheorem{note}[thm]{Note}

\newcommand {\tb}{\textbf}
\newcommand {\mb}{\mathbb}
\newcommand {\Z}{\mb Z}
\newcommand {\R}{\mb R}

\newcommand {\im}{\textrm{im}}
\newcommand {\colim}{\textrm{colim}\ }

\newcommand {\ex}{\mathrm{excess}}
\newcommand {\ext}{\mathrm{Ext}}
\newcommand {\lra}{\longrightarrow}
\newcommand {\la}{\langle}
\newcommand {\ra}{\rangle}

\begin{document}

\title{On the Hurewicz homomorphism on the extensions of ideals in $\pi_*^s$ and spherical classes in $H_*Q_0S^0$}

\author{Hadi Zare\\
        School of Mathematics, Statistics,
        and Computer Sciences\\
        University of Tehran, Tehran, Iran\\
        \textit{email:hadi.zare} at \textit{ut.ac.ir}}
\date{}

\maketitle

\begin{abstract}
This is about Curtis conjecture on the image of the Hurewicz map $h:{_2\pi_*}Q_0S^0\to H_*(Q_0S^0;\Z/2)$. First, we show that if $f\in{_2\pi_*^s}$ is of Adams filtration at least $3$ with $h(f)\neq 0$ then $f$ is not a decomposable element in ${_2\pi_*^s}$. Moreover, it is shown if $k$ is the least positive integer that $f$ is represented by a cycle in $\ext^{k,k+n}_A(\Z/2,\Z/2)$, then (i) if $e_*h(f)\neq 0$ then $n\geqslant 2^k-1$; (ii) if $e_*h(f)=0$ then $n\geqslant 2^k-2^t$ for some $t>1$. Second, for $S\subseteq{_2\pi_{*>0}^s}$ we show that: (i) if the conjecture holds on $S$, then it holds on $\la S\ra$; (ii) if $h(S)=0$ then $h$ acts trivially on any extension of $S$ obtained by applying homotopy operations arising from ${_2\pi_*}D_rS^n$ with $n>0$. We also provide partial results on the extensions of $\la S\ra$ by taking (possible) Toda brackets of its elements. We also discuss how the $EHP$-sequence information maybe applied to eliminate classes from being spherical.
\end{abstract}

\tableofcontents

\section{Introduction and statement of results}
This note is circulated around Curtis conjecture; some of the observations here might be well known, but we don't know of any published account. Let $Q_0S^0$ be the base point component of $QS^0=\colim \Omega^iS^i$ corresponding to stable maps $S^0\to S^0$ of degree $0$. The conjecture then reads as follows (see \cite[Proposition 7.1]{Curtis} and \cite{Wellington} for more discussions).
\begin{mconj}[The Curtis Conjecture]\label{conjecture}
In positive degrees only the Hopf invariant one and Kervaire invariant one elements survive under the Hurewicz homomorphism $h:{_2\pi_*}Q_0S^0\lra H_*(Q_0S^0;\Z/2)$.
\end{mconj}

Note that given $f,g\in{_2\pi_*}Q_0S^0$ with $h(f)\neq 0$ and $h(g)=0$ then $h(f+g)=h(f)\neq 0$. Also, note that if $\alpha:S^0\not\lra S^0$ is any map of odd degree, then $h(\alpha f)=h(f)$. Finally, note that the Hurewicz homomorphisms ${\pi_*}Q_0S^0\to H_*(Q_0S^0;\Z)$ and ${_2\pi_*}Q_0S^0\to H_*(Q_0S^0;\Z/2)$ are $\Z$-module homorphisms, and not ring homomorphisms (see Theorem \ref{main0} below). These hopefully will justify the way that we have stated our results. The above conjecture has a variant for path connected CW-complexes due to P. J. Eccles. For $X$, let $QX=\colim \Omega^i\Sigma^iX$. We may state the conjecture as follows.

\begin{mconj}[The Eccles Conjecture]\label{Ecclesconjecture}
Suppose $X$ is a path connected CW-complex of finite type. If $\alpha\in{_2\pi_nQX}$, $n>0$, maps nontrivially under the Hurewicz homomorphism $h:{_2\pi_*}QX\lra H_*(QX;\Z/2)$ then $\alpha$ is either stable spherical or its stable adjoint is detected by a primary operation in its (stable) mapping cone. Here, stably spherical means that the stable adjoint of $\alpha$, $S^n\to X$ is detected by homology.
\end{mconj}

A schematic diagram for the relation between the two conjectures may be state as follows
$$\begin{array}{ccc}
\textrm{Curtis conjecture}&\Rightarrow&\left.\begin{array}{c}\textrm{Eccles conjecture for}\\X=S^n, n>0\end{array}\right.\\ \\
\left.\begin{array}{c}\textrm{Eccles conjecture for}\\X=\R P\end{array}\right.&\Rightarrow&\textrm{Curtis conjecture}
\end{array}$$
where the passage from the Eccles conjecture to the Curtis conjecture is throughout the Kahn-Priddy theorem \cite[Theorem 3.1]{Kahn-Priddy} and Lin's variant of it \cite[Theorem 1.1]{Lin}; the passage from Curtis conjecture to Eccles conjecture is through the fact the Kervaire invariant one elements map to square terms in $H_*(Q_0S^0;\Z/2)$, and consequently are killed by homology suspension, and only Hopf invariant one elements do survive.\\

Here and throughout the paper, we write $\la S\ra$ for the ideal generated by a set $S$ in a ring $R$, $\pi_*$ and $\pi^s_*$ for homotopy and stable homotopy respectively, and ${_p\pi_*},{_p\pi_*^s}$ for their $p$-primary components, respectively. We write $\pi_{*>0}^s$ for $\oplus_{i>0}\pi_i^s$. We shall use $f_\#$ to denote the mapping induced in homotopy, and $f_*$ for the mapping induced in homology, where $f$ is a mapping of spaces or stable complexes, and $C_f$ will refer to (stable) mapping cone of $f$. To avoid confusion, we write $\int$ for summation, $\Sigma$ for the suspension functor on the category of pointed spaces as well as spectra, and $\Sigma_*:\widetilde{H}_nX\to H_{n+1}\Sigma X$ for the suspension isomorphism. Finally, $\eta,\nu,\sigma$ will denote the well known Hopf invariant one elements. We also write $e$ for the evaluation map, and $e_*$ for the homology suspension, so that it is not confused with the homology of the Hopf map $\sigma\in{\pi_7^s}$. Our fist observation is the following.

\begin{mthm}\label{main0}
(i) For $i,j>0$, consider the composition
$${\pi_i}Q_0S^0\otimes {\pi_j}Q_0S^0\lra {\pi_{i+j}}Q_0S^0\stackrel{h}{\lra} H_{i+j}(Q_0S^0;\Z)$$
where the first arrow is the product in $\pi_*^s$. Then $h(fg)\neq 0$ only if $f$ and $g$ live in the same grading, and both are detected by the unstable Hopf invariant.\\
(ii) For $i,j>0$, consider the composition
$${_2\pi_i}Q_0S^0\otimes {_2\pi_j}Q_0S^0\lra {_2\pi_{i+j}}Q_0S^0\stackrel{h}{\lra} H_{i+j}(Q_0S^0;\Z/2).$$
Then $h(fg)\neq 0$ only if $f=g$ with $f=\eta,\nu,\sigma$ or odd multiples of these elements, i.e. the image of this composition only consists of Kerviare invariant one elements $h(\eta^2),h(\nu^2),h(\sigma^2)$. Here, the first arrow on the left is the multiplication on the stable homotopy ring. Also, the image of the composite
$$\la f:f=\eta,\nu,\sigma\ra\rightarrowtail{_2\pi_*}Q_0S^0\stackrel{h}{\lra}H_*(Q_0S^0;\Z/2)$$
only consists of the Hurewicz image of the Hopf invariant one elements $\eta,\nu,\sigma$, and the Kervaire invariant one elements $\eta^2,\nu^2,\sigma^2$.\\
(iii) Suppose $E$ is a $CW$-spectrum, and let $\Omega^\infty E=\colim\Omega^nE_n$. Then for the composition
$$h^\circ:\pi_iQ_0S^0\otimes\pi_j\Omega^\infty E\to\pi_{i+j}\Omega^\infty E\lra H_*(\Omega^\infty E;\Z)$$
with $i>0$, we have
$$\begin{array}{ccl}
i<j &\Rightarrow& h^\circ(\alpha\circ f)=0,\\
i>j &\Rightarrow& e_*^{i-j+1}h^\circ(\alpha\circ f)=0.
\end{array}$$
Moreover, if $i=j$ and $h^\circ(\alpha\circ f)\neq 0$ then $\alpha$ is detected by the unstable Hopf invariant; in particular at the prime $2$, $\alpha\in\{\eta,\nu,\sigma\}$.\\
Here, $\circ$ is the product coming from the composition of stable maps $S^{i+j}\to S^j$ and $S^j\to X$ turning $\pi_*\Omega^\infty E$ into a left $\pi_*^s$-module.
\end{mthm}

Note that a part of the statement of $(iii)$ stays valid if we replace $H\Z/2$-homology with any generalised homology $R$, obtained from a ring spectrum with unit, for the $R$-Hurewicz homomorphism
$$h_R:\pi_*\Omega^\infty E \lra R_*\Omega^\infty E$$
if any $f:S^{n+i}\to S^i$, with $n>0$, induces a trivial map in $R_*$-homology, at least after localisation at $2$. At least case $i<j$ of $(iii)$ stays valid for any such homology theory. There exists example of computing Hurewicz image in some `interesting' cases (for example see \cite{Snaith-BP}) but we don't know if they have been interpreted in terms of homology of infinite loop spaces. Our first application of this theorem is to provide some necessary conditions for an element $f\in{_2\pi_*^s}$ to map nontrivially under $h$ which we state as follows.
\begin{mthm}\label{necessary}
Suppose $f\in{_2\pi_*^s}$ is of Adams filtration $>2$, and $h(f)\neq 0$. Then, $f$ is not a decomposable element in ${_2\pi_*^s}$.
\end{mthm}

The proof of the above theorem is very short, so we include it here.

\begin{proof}
Suppose $f$ is a sum of decomposable terms. Since $f$ is Adams filtration at least $3$, then if written as a sum of decomposable terms in ${_2\pi_*^s}$, it cannot involve terms such as $\eta^2,\nu^2$ or $\sigma^2$. Hence, by Theorem \ref{main0}, $h(f)=0$.
\end{proof}

The above theorem has a variant on the level of ASS due to Hung and Peterson \cite[Proposition 5.4]{Hung-Peterson}. It states that for the Lannes-Zarati homomorphism $\varphi_k:\mathrm{Ext}^{k,k+i}_A(\Z/2,\Z/2)\to (\Z/2\otimes D_k)_i^*$, which is meant to filter the Hurewicz homomorphism $\pi_iQ_0S^0\to H_i(Q_0S^0;\Z/2)$ on the $E_\infty$-page level of the ASS, $\varphi_k$ does vanish on decomposable classes in the spectral sequence when $k>2$. Unfortunately, the relation between $\oplus_k\varphi_k$ and $h$ is not of that linear type, and it is not clear to the author whether or not if one of these results implies the other. Although, we like to note that the two conditions of being a decomposable permanent cycle in ASS and representing a decomposable element in ${_2\pi_*^s}$ are not necessarily the same; for instance $h_j^2$ in ASS with $3<j<7$ represents a Kervaire invariant one element $\theta_j$ where by Theorem \ref{main0} the latter is not a decomposable in ${_2\pi_*^s}$.\\

Next, together with some detailed calculations on $h(\kappa)$, the following becomes an immediate corollary of Theorem \ref{main0}.

\begin{mthm}\label{ideal}
Suppose that Curtis conjecture holds on $S\subset{_2\pi_{*>0}^s}$. Then the conjecture also holds on $\la S\ra$.
\end{mthm}

An \textit{ad hoc} way to verify the conjecture, then is try to choose the set $S$ as large as possible by feeding in more elements into $S$. For instance, writing $J$ for the fibre of the Adams operation $\psi^3-1:\mathrm{BSO}\lra \mathrm{BSO}$, it is known that the conjecture holds on ${_2\pi_*}J$ \cite[Theorem 1]{Za}. Let $F\subset{_2\pi_{*>0}^s}$ be the set of all the elements of ${_2\pi_*^s}$ which are detected in either $1$-line or $2$-line of ASS, and write $\tau_i\in{_2\pi_{2^{i+1}+1}^s}$ for Bruner's family which live in the $3$-line of the ASS. The following provides an example for applying Theorem \ref{ideal}.

\begin{mlmm}\label{main1}
Consider the composition
$$\la f:f\in F+{_2\pi_*}J,\textrm{ or }f=\tau_i\ra\rightarrowtail{_2\pi_*^s}\simeq{_2\pi_*}Q_0S^0\lra H_*(Q_0S^0;\Z/2).$$
Then, only elements of Hopf invariant one or Kervaire invariant one map nontrivially under the above composition.
\end{mlmm}

Let's note that recently Gaudens \cite[Theorem 6.5]{Ga} has used a BV-algebra structure on the homology of double loop spaces to verify Theorem \ref{ideal} in the case $S$ is the image of $J$-homomorphism ${_2\pi_*}\mathrm{SO}\lra{_2\pi_*}Q_0S^0$. Since, at the prime $2$, $\im((J_\R)_\#)\subset{_2\pi_*}J$, the above Lemma provides a generalisation of Gaudens' results, where we have made no use of string structures. \\

Next, note that for a given set $S$ in a ring, the ideal $\la S\ra$ is the largest that we can get from the ring multiplication. Hence, we have to look for other methods for extending $S$ or $\la S\ra$ into a larger set. Working in ${_2\pi_*^s}$, we consider Toda brackets and homotopy operations. First we consider Toda brackets. An ideal result would be to show that ``if the conjecture hold on $S$, then it also holds on the extension of $S$, or of $\la S\ra$, by all (higher) Toda brackets'' of course modulo our definition of a higher bracekt! By a result of J. Cohen \cite[Thoerem 4.5]{JCohen}, any element of ${_2\pi_*^s}$ might be written as a (higher) Toda bracket, as he defined, of $2,\eta,\nu,\sigma$. Hence, such a statement would imply the Conjecture \ref{conjecture} as by Cohen's result it is enough to choose $S=\{2,\eta,\nu,\sigma\}$. We don't have such a statement. In this direction, $(1)$ we have obtained some necessary conditions for nonvanishing of the Hurewicz image of elements represented by Toda bracket; $(2)$ we have observed that under expected stability conditions the Hurewicz image does indeed vanish.

\begin{mprp}\label{tripleToda}
(i) Let $\alpha\in{_2\pi_i^s}X_1$, $\beta:X_1\to X_2$ and $\gamma:X_2\to S^0$ be stable maps with $X_1$ and $X_2$ being stable complexes of finite type. Suppose that the Toda bracket $\{\alpha,\beta,\gamma\}$ is defined and represents $f\in{_2\pi_*^s}$ such that $h(f)\neq 0$. Then, at least one of $\gamma:X_2\to QS^0$ or $\alpha:S^i\to QX_1$ is nontrivial in homology.\\
(ii) Suppose $\alpha\in\pi_i^sX_1$. Let $\beta:X_1\to X_2$, $\gamma:X_2\to S^0$ be stable maps, with $X_1,X_2$ being (stable) complexes, $i<2\mathrm{conn}(X_1)$, $\dim X_1<2\mathrm{conn}(X_2)$, $i>\dim X_1>\dim X_2$, so that the Toda bracket $\{\alpha,\beta,\gamma\}$ is defined. Here, $\mathrm{conn}(X)=n$ if $\pi_t^sX\simeq 0$ for all $t<n$. Then, the element of $\pi_{i+1}Q_0S^0$ represented by Toda bracket $\{\alpha,\beta,\gamma\}$ maps trivially under the Hurewicz homomorphism $h:\pi_{i+1}Q_0S^0\to H_{i+1}(Q_0S^0;\Z)$. In particular, this is true on the $2$-component of $\pi_*^s$.
\end{mprp}

We have a few comments in order. First, the idea here is that higher Toda brackets of maps among (wedge of) spheres, can be written as a triple Toda bracket among complexes so that the complexes somehow encode some of the maps in the higher bracket. Second, note that for the case $(i)$ of the above theorem, when we choose $X_1$ and $X_2$ to be suspensions of sphere spectrum, then the condition $\alpha:S^i\to X_1$ is stronger than mapping nontrivially under $h:{_2\pi_*}Q_0S^0\to H_*(Q_0S^0;\Z/2)$. More precisely, let $\alpha\in\pi_i^s$, $\beta\in\pi_j^s$ and $\gamma\in\pi_k^s$ such that the stable composition $\beta\alpha$ and $\gamma\beta$ are trivial. The we may consider
$$S^{i+j+k}\stackrel{\alpha}{\to} QS^{j+k}\stackrel{\beta}{\to}QS^k\stackrel{\gamma}{\to}QS^0$$
to form a Toda bracket representing an element in $\pi_{i+j+k+1}QS^0\simeq\pi_{i+j+k+1}^s$. Now, by the above proposition, if $\{\alpha,\beta,\gamma\}$ is defined and maps nontrivially into $H_*(QS^0;\Z/2)$ then we need $\alpha:S^{i+j+k}\to QS^{j+k}$ and $\gamma:S^k\to QS^0$ to have nontrivial homology, which for $\alpha:S^{i+j+k}\to QS^{j+k}$ is a stronger requirement than just mapping nontrivially under the Hurewicz map $\pi_iQS^0\to H_i(QS^0;\Z/2)$. To see this, note that the Kerviare invariant one elements $\theta_j\in{_2\pi_{2^{j+1}-2}^s}$ map to square terms $p_{2^{i}-1}^2\in H_*(Q_0S^0;\Z/2)$ where $p_{2n+1}$'s are certain primitive elements in $H_*(Q_0S^0;\Z/2)$ \cite[Theorem 7.3]{Madsenthesis}. This implies that $e_*h(\theta_j)=0$ in $H_*(QS^1;\Z/2)$ and consequently, for any for any nonzero value of $j,k$, $\alpha=\theta_j$ does not fulfill this condition. Hence, for any choice of $\beta$ and $\gamma:S^k\to QS^0$ with $\gamma_*=0$, the Toda bracket $\{\theta_j,\beta,\gamma\}$, if defined, maps trivially into $H_*(Q_0S^0;\Z/2)$.

\begin{mcrl}\label{Todaextension1}
$(i)$ Suppose $S\subset{_2\pi_{*>0}^s}$ such that for any $\alpha:S^i\to Q_0S^0$ belonging to $S$ either $\alpha_*=0$ or $e_*h(\alpha)=0$. Then, for any $\alpha\in S$, $\beta\in{_2\pi_j^s}$, and $\gamma:S^k\to Q_0S^0$ with $\gamma_*=0$, the Toda bracket $\{\alpha,\beta,\gamma\}$ maps trivially into $H_*(Q_0S^0;\Z/2)$.\\
$(ii)$ Suppose $S\subset{_2\pi_{*>0}^s}$ such that the Curtis conjecture holds on it. Then, the conjecture hold on the extensions of $\la S\ra$ obtained by triple Toda brackets $\{\alpha,\beta,\gamma\}$, with $\alpha\in S$, $\beta\in{_2\pi_*^s}$ is arbitrary, such that either one of the following conditions hold:\\
$(1)$ $\gamma_*=0$;\\
$(2)$ $\gamma_*\neq 0$ and $\alpha:S^{i+j+k}\to QS^{j+k}$ is nontrivial in homology.
\end{mcrl}

An examples of a set $S$ for the case $(i)$ of the above corollary is the ideal considered in Lemma \ref{main1}. The above corollary leaves the cases that $\alpha:S^{i+j+k}\to QS^{j+k}$ is trivial in homology, but $\gamma:S^k\to QS^0$ is not. The proof involves many detailed calculations, some of which we believe to be new.\\
Next, we consider the possible extensions by available homotopy operations coming from $\pi_*^sD_rS^n$ for $n>0$ where $D_r=E\Sigma_r\ltimes_{\Sigma_r}(-)^{\wedge r}$, with $\Sigma_r$ being the permutation group on $r$ elements, is the $r$-adic construction. We have the following which is more or less expected.
\begin{mprp}
Suppose $\alpha\in{_2\pi_m^s}D_rS^n$ and $f\in{_2\pi_n}Q_0S^0$ such that $h(f)=0$. Then, for $\alpha^*(f)$ defined by the composition
$$S^m\stackrel{\alpha}{\lra}D_rS^n\stackrel{D_rf}{\lra}D_rS^0\stackrel{\mu}{\lra}S^0$$
with $\mu:D_rS^0={B\Sigma_r}_+\to S^0$ being induced by the $H_\infty$ ring structure of $S^0$, we have $h(\alpha^*(f))=0$ where $h:{_2\pi_*}Q_0S^0\lra H_*(Q_0S^0;\Z/2)$ is the Hurewicz homomorphism.
\end{mprp}

\begin{mrmk}
Note that, ideally, we would like to show that if the conjecture holds on $S\subseteq{_2\pi_{*>0}}$, then conjecture holds on the extension of $S$ by (iterated) application of homotopy operations available throughout ${_2\pi_*^s}D_rS^n$. By the above theorem if $h(\alpha^*(f))\neq 0$ then $h(f)\neq 0$. If $f\in S$ on which the conjecture holds, then $f$ is either a Hopf invariant one or a Kervaire invariant one element. The proof will be complete if we show that for any $\alpha$, $h(\alpha^*(\theta_j))=0$, and $h(\alpha^*(h_i))=0$ unless $\alpha^*(h_i)$ is a Hopf or Kerviare invariant one element. Here, we have used $h_i$ for one of the Hopf invariant one elements, and $\theta_j$ for Kervaire invariant one elements. We do not have a proof of this latter claim yet.
\end{mrmk}

Next, we consider the relation between dimension of a spherical class and the Adams filtration of homotopy classes that map to it (compare to \cite[Corollary 1.5]{Ku}). We have the following.

\begin{mthm}
Suppose $f\in{_2\pi_n}Q_0S^0$ of Adams filtration $k$, i.e. $k$ is the least positive integer where $f$ is represented by a cycle in $\ext^{k,k+n}_A(\Z/2,\Z/2)$, and  $h(f)\neq 0$ where $h:{_2\pi_n}Q_0S^0\to H_n(Q_0S^0;\Z/2)$ is the Hurewicz homomorphism. The following statement then hold.\\
(i) If $e_*h(f)\neq 0$ then $n\geqslant 2^k-1$.\\
(ii) If $e_*h(f)=0$, $h(f)=\xi^{2^t}$ with $e_*\xi\neq 0$, and $h(f)\in\im(\varphi_k)$ then $n\geqslant 2^{k}-2^t$. Here, $\varphi_k$ denotes the Lannes-Zarati homomorphism.
\end{mthm}

\section{Preliminaries}

\subsection{Iterated loop spaces}\label{iteratedloop}
We wish to recall some standard facts on iterated loop spaces, and refer the reader to \cite{May-G} and \cite{Ad} for more details. We refer to a space $X$ as a $n$-fold loop space, or $\Omega^n$-space for short, with $n\leqslant +\infty$, if there exists a collection of spaces $X_i$, $i=0,1,2,\ldots,n$, together with homotopy equivalences $X_i\to\Omega X_{i+1}$ such that $X=X_0$. Obviously, an $\Omega^n$-space is also an $\Omega^l$-space for $l<n$. An $\Omega^n$-space $X$ is an $E_n$-algebra in the operadic language of \cite{May-G} and admits a `structure map' or `evaluation map' $\theta_n(X):\Omega^n\Sigma^nX\to X$, briefly denoted by $\theta_n$ if there is no confusion; the structure map itself is an $n$-fold loop map. The spaces of the form $\Omega^n\Sigma^nX$ are ought to play the role of free objects in the category of $E_n$-algebras so that any map $f:Y\to X$ with $X$ being an $\Omega^n$-space, admits a unique extension to a $n$-fold loop map $\Omega^n\Sigma^nY\to X$ defined by the composite
$$\Omega^n\Sigma^nY\stackrel{\Omega^n\Sigma^nf}{\lra}\Omega^n\Sigma^nX\stackrel{\theta_n}{\lra}X.$$
In the case of $\Omega^\infty$-spaces, our main focus will be spaces $QX=\colim\Omega^n\Sigma^nX$ which by definition satisfy $\Omega Q\Sigma X=QX$.\\
\textbf{Adjoints.} The functors $\Sigma^n$ and $\Omega^n$, $n>0$, as well as $\Sigma^\infty$ and $\Omega^\infty$ are adjoint functors. For $f:\Sigma X\to\Sigma Y$, a map of spaces, we have its adjoint as a map $X\to \Omega\Sigma Y$. It follows from the definition that the adjoint mapping can also be written as the composite
$$X\stackrel{E_X}{\to}\Omega\Sigma X\stackrel{\Omega f}{\to}\Omega\Sigma Y$$
where $E_X:X\to\Omega\Sigma X$ is the suspension map, being adjoint to the identity $\Sigma X\to\Sigma X$. It is then obvious that the extension of both maps to $\Omega$-maps $\Omega\Sigma X\to \Omega\Sigma Y$ are the same and agree with $\Omega f:\Omega\Sigma X\to\Omega\Sigma Y$. Similar statement holds for maps $f:\Sigma^nX\to\Sigma^nY$ with $n<+\infty$. Moreover, if $X\to Y$ is a map of suspensions spectra, then its stable adjoint $X\to QY$ extends to an infinite loop map $QX\to QY$ by the above construction. On the other hand, we may apply $\Omega^\infty$ to the above map to obtain an infinite loop map $QX\to QY$. The two constructions, indeed yield the same map.

\subsection{Homology of iterated loop spaces}
We recall a description of $\Z/2$-homology of $\Omega^{n+1}$-loop spaces of the form $\Omega^{n+1}\Sigma^{n+1}X$. Let $Y$ be an $\Omega^{n+1}$-space, $0\leqslant n\leqslant+\infty$ with the convention $+\infty+1=+\infty$. Since $Y$ is a loop space, the homology $H_*(Y;\Z/2)$ then is a ring under Pontrjagin product. Moreover, there are group homomorphisms \cite[Part I, Theorem I]{CLM} and \cite[Part III, Theorem 1.1]{CLM} (see also \cite{Kuhnhomology})
$$Q_i:H_d(Y;\Z/2)\lra H_{i+2d}(Y;\Z/2)$$
for $i<n+1$ with $Q_0$ acting as the squaring operation with respect to the Pontrjagin product, i.e. $Q_0\xi=\xi^2$. These operations fit into an algebra known as the Dyer-Lashof algebra, often denoted by $R$. For a sequence $E=(e_1,\ldots,e_s)$ we may abbreviate $Q_{e_1}\cdots Q_{e_s}$ to $Q_E$. For a path-connected space $X$, the homology ring $H_*(\Omega^{n+1}\Sigma^{n+1}X;\Z/2)$ is a ring, and as a module of $R$ in the case of $n=+\infty$, is described by (see \cite[Part III, Lemma 3.8]{CLM}, \cite[Part I, Lemma 4.10]{CLM})
$$H_*(\Omega^{n+1}\Sigma^{n+1}X;\Z/2)\simeq\Z/2[Q_Ex_\mu:E\textrm{ nondecreasing, }e_1>0, e_s<n+1]$$
where $\{x_\mu\}$ is an additive basis for the reduced homology $\widetilde{H}_*(X;\Z/2)$, i.e. symbols as $Q_Ex_\mu$ are generators of this polynomial algebra. We allow the empty sequence $\phi$ to be nondecreasing with $Q_\phi\xi=\xi$. Sometimes, it is more convenient to work with upper indexed operations, $Q^i:H_d(\Omega^{n+1}Y;\Z/2)\to H_{d+i}(\Omega^{n+1}Y;\Z/2)$, known as the Kudo-Araki operations, defined by $Q^i\xi=Q_{i-d}\xi$. For $I=(i_1,\ldots,i_s)$, we say $I$ is admissible if $i_j\leqslant 2i_{j+1}$, and the excess is defined by $\ex(Q^Ix_\mu)=i_1-(i_2+\cdots+i_s+\dim x_\mu)$. We allow $\phi$ to be admissible with $Q^\phi\xi=\xi$ and $\ex(Q^\phi x_\mu)=+\infty$. Note that if $I$ is admissible for $Q^Ix_\mu=Q_Ex_\mu$ then $E$ is nondecreasing and vice versa; also $\ex(Q^Ix_\mu)=e_1$. In the case of the infinite loop space $QX$, for a path connected space $X$, using the upper indexed operations we have
$$H_*(QX;\Z/2)\simeq\Z/2[Q^Ix_\mu:I\textrm{ is admissble }, \ex(Q^Ix_\mu)>0].$$
The operations $Q^i$ are additive homomorphisms with $Q^ia=a^2$ if $i=\dim a$ and $Q^ia=0$ if $i<\dim a$; hence $\ex(Q^Ix_\mu)=0$ means that $Q^Ix$ is a square in the polynomial ring $H_*(QX;\Z/2)$.\\
Write $X_+$ for $X$ with a disjoint base point. If $X$ is path connected, we may describe homology of $Q_0(X_+)$, the base point component of $Q(X_+)$, as follows. Write $[n]$ for the image of $n\in\pi_0Q(X_+)\simeq\pi_0^s(X_+)\simeq\Z$ in $H_0(Q(X_+);\Z)$ under the Hurewicz homomorphism. Then, we have
$$H_*(Q_0(X_+);\Z/2)\simeq\Z/2[Q^Ix_\mu*[-2^{l(I)}]:I\textrm{ is admissble }, \ex(Q^Ix_\mu)>0]$$
where $*$ is the Pontrjagin product in $H_*(Q(X_+);\Z/2)$ induced by the loop sum. Note that $Q^Ix_\mu*[-2^{l(I)}]$ is not a decomposable in $Q_0(X_+)$ whereas it is in $Q(X_+)$.\\
If $f:X\to Y$ is given, we then obtain a map of $\Omega^{n+1}$-spaces, $n\leqslant+\infty$, as $\Omega^{n+1}\Sigma^{n+1}f:\Omega^{n+1}\Sigma^{n+1} X\to \Omega^{n+1}\Sigma^{n+1}Y$. The homology of this map on generators $Q^Ix$ is determined by
$$(\Omega^{n+1}\Sigma^{n+1}f)_*Q^Ix=Q^If_*x.$$
Similarly, for the homology of the induced map $Q(X_+)\to Q(Y_+)$ we have
$$(Q(f_+))_*(Q^Ix*[-2^l(I)])=Q^If_*x*[-2^{l(I)}].$$
The action of the Steenrod algebra on the generators $H_*(\Omega^{n+1}\Sigma^{n+1}X;\Z/2)$, on the generators, is determined by (iterated application of) Nishida relations
\begin{equation}\label{Nishida}
Sq^a_*Q^b=\int_{t\geqslant 0}{b-a\choose a-2t}Q^{b-a+r}Sq^t_*
\end{equation}
where $Sq^i_*:H_*(-;\Z/2)\to H_{*-i}(-;\Z/2)$ is the operation dual to $Sq^i:H^*(-;\Z/2)\to H^{*+i}(-;\Z/2)$. Moreover, for $\xi,\zeta\in H_*(\Omega^{n+1}\Sigma^{n+1}X;\Z/2)$, there is a Catran formula \cite[Remark 1.9]{Wellington} that
$$Sq^t_*(\xi\zeta)=\int_{i=0}^t (Sq^{t-i}_*\xi)(Sq^i_*\zeta).$$
These relations completely determine the action of Steenrod algebra on homology of $\Omega^{n+1}\Sigma^{n+1}X$ and $Q_0(X_+)$.

\subsection{Homology suspension}
Let $e:\Sigma\Omega X\to X$ be the adjoint of $1_{\Omega X}:\Omega X\to \Omega X$. The homology suspension $\sigma_*$ is defined as
$$H_*\Omega X\stackrel{\simeq}{\lra} H_{*+1}\Sigma\Omega X\stackrel{e_*}{\lra}H_{*+1}X.$$
Since $OX=\Omega Q\Sigma X$ we then may consider the homology suspension $e_*:H_nQX\lra H_{n+1}Q\Sigma X$; in this case $e_*$ is characterised by the following properties: (1) $\sigma_*Q^Ix=Q^I\Sigma_*x$ for $x\in \widetilde{H}_*X$; (2) $\sigma_*\xi=0$ if $\xi$ is a decomposable class in the polynomial algebra $H_*QX$. Moreover, by \cite[Page 47]{CLM} the homology suspension $H_*(Q_0(X_+);\Z/2)\lra H_{*+1}(Q\Sigma(X_+);\Z/2)$ is characterised by the following properties: (1) $\sigma_*$ acts trivially on decomposable terms; (2) on the generators it is given by
$$e_*(Q^Ix_\mu*[-2^{l(I)}])=Q^I\Sigma_*x_\mu.$$
Similar statements hold for the homology suspension $e_*:H_*\Omega^{n+1}\Sigma^{n+1}X\to H_{*+1}\Omega^n\Sigma^{n+1}X$. \\
For a Hopf algebra $H$, we write $DH$ for the submodule of decomposable elements and $\mathrm{Ind}(H)=H/DH$ for the quotient module of indecomposable elements. The following provides more information on the action of homology suspension.

\begin{lmm}\label{kernelofsuspension}
Let $X$ be a path connected space.\\
(i) The homology suspension $e_*:H_nQX\lra H_{n+1}Q\Sigma X$ does induce a monomorphism $\sigma_*:\mathrm{Ind}(H_nQX)\lra H_{n+1}Q\Sigma X$. In particular, if $\int Q^Ix_\mu\in\ker e_*$ where sum runs over some $I,x_\mu$ then $Q^Ix_\mu\in\ker e_*$ for all $I,x_\mu$ in the sum, i.e. $\ex(Q^Ix_\mu)=0$.\\
(ii) The homology suspension $e_*:H_nQ_0(X_+)\lra H_{n+1}Q\Sigma(X_+)$ does induce a monomorphism $\sigma_*:\mathrm{Ind}(H_nQ_0(X_+))\lra H_{n+1}Q\Sigma(X_+)$. In particular, if $\int Q^Ix_\mu*[-2^{l(I)}]\in\ker e_*$ where sum runs over some $I,x_\mu$ then $Q^Ix_\mu*[-2^{l(I)}]\in\ker e_*$ for all $I,x_\mu$ in the sum, i.e. $\ex(Q^Ix_\mu)=0$.
\end{lmm}

\begin{proof}
(i) Let $\xi\in\mathrm{Ind}(H_nQX)$, a nonzero element, be represented by a nonempty sum $\int Q^Ix_\mu$ of monomial generators of $H_*QX$ with $\dim(Q^Ix_\mu)=n$ and $\ex(Q^Ix_\mu)>0$, modulo decomposable terms. Hence, as decomposables are killed by $e_*$, we have
$$e_*\xi=e_*\int Q^Ix_\mu=\int Q^IE_*x_\mu$$
where every term on the right satisfies $\ex(Q^I\Sigma_*x_\mu)\geqslant 0$, being distinct monomials in the polynomial algebra $H_*Q\Sigma X$ since $Q^Ix_\mu$'s were distinct in $H_*QX$. Hence, $e_*\xi\neq 0$ in $H_*Q\Sigma X$ which verifies that $\sigma_*:\mathrm{Ind}(H_nQX)\lra H_{n+1}Q\Sigma X$ is a monomorphisms. In particular, the equation $\sigma_*\xi=\int Q^I\Sigma_*x_\mu$ shows that a sum $\int Q^Ix_\mu\in\ker\sigma_*$ only if all $Q^Ix_\mu\in\ker\sigma_*$. This completes the proof. The proof of (ii) is similar.
\end{proof}
Finally, we note that the homology of a map $f:X\to \Omega Y$ and its adjoint $\widetilde{f}:\Sigma X\to Y$ are related through $e_*$ by the coomutative diagram
$$\xymatrix{H_*X\ar[r]^-{\widetilde{f}_*}\ar[d]_-{\Sigma_*(\simeq)} & H_*\Omega Y\ar[d]^-{e_*}\\
            H_{*+1}\Sigma X\ar[r]^-{f_*}                  & H_{*+1}Y.}$$

\subsection{Spherical classes}
Let $k=\Z,\Z/p$ with $p$ some prime number. A homology class $\xi\in H_n(X;k)$ is called spherical if it is in the image of the Hurewicz homomorphism $\pi_nX\to H_n(X;k)$. A spherical class $\xi$ has some nice properties which follow from its definition and the basic properties of $H_*(S^n;k)$: (1) $\xi$ is a primitive in the coalgebra $H_*(X;k)$ where the coproduct is induced by the diagonal $X\to X\times X$; (2) $\xi$ pulls back to a spherical class $\xi_{-1}\in H_{*-1}(\Omega X;k)$ which is not necessarily unique; (3) when $k=\Z/p$ then $\xi$ is annihilated by all Steenrod operations $P^i_*$; we are mainly interested in the case $p=2$ in which the property $(3)$ reads as $Sq^i_*\xi=0$ for all $i>0$.

\subsection{Stable splitting of $QX$}
According to May \cite{May-G} for a space $X$, $QX$ maybe approximated by a free monoid $CX$. The monoid $CX$ also can be obtained throughout simplicial methods as shown by Barratt and Eccles \cite{BE1},\cite{BE}, \cite{BE3}; in their notation $\Gamma^+$ stands for $C$ and more precisely for a simplicial space $X$ there is a weak equivalence between $C|X|$ and $|\Gamma^+ X|$ where $|\ |$ is the geometric realisation functor. We henceforth use $C$ and $\Gamma^+$  interchangeably when we refer to the literature. The approximation then is described as following. There exists a space $CX$, which is a free monoid object \cite[Proposition]{BE1}, together with an inclusion $i:CX\to QX$ which is a weak equivalence when $X$ is path connected, and a group completion in general \cite[Proposition 5.1, Corollary 5.4]{BE}. A construction of group completion which is popular is the $\Omega B$ functor, with $B$ being the classifying space functor, applied to associative and not necessarily commutative, monoids \cite[Proposition 2.3.4]{BarrattPriddy} (see also \cite[Theorem 3.2.1]{Ad}). Hence, the approximation theorem says that when $X$ is not path connected, there is a homotopy equivalence $\Omega B(CX)\to QX$; in particular, noting that $D_rS^0={B\Sigma_r}_+$, this yields the famous Barratt-Priddy-Quillen theorem that $\Omega B(B\Sigma_\infty)\simeq QS^0$ in the case of $X=S^0$. Moreover, by construction of the universal group \cite{BarrattPriddy}, $\Omega B$, it follows that for a monoid $M$ the natural inclusion $M\to\Omega BM$ is a monomorphism in homology; in particular $i:CX\to QX$ induces a monomorphism in homology.\\
There exist a canonical inclusion $\iota_X:X\to CX$, and a structure map (or evaluation map) $\theta_{CX}:CCX\to CX$ which turn $(C,\iota,\theta)$ into a monad \cite[Proposition 3.6]{BE1}. The space $CX$ admits a filtration $\{F_rCX\}$ such that the successive quotients $F_rCX/F_{r-1}CX$ are the same as the $r$-adic construction on $X$, $D_rX=E\Sigma_r\ltimes_{\Sigma_r}X^{\wedge r}$, with $D_1X=X$, as introduced in Section 1. For an arbitrary compactly generated space $X$, which all of the spaces in this paper are of that type, we have the well-known Snaith splitting \cite[Theorem 1.1]{Snaith}
$$\Sigma^\infty CX\stackrel{\simeq}{\lra}\bigvee_{r=1}^{+\infty}\Sigma^\infty D_rX$$
(see also \cite[Theorem C]{BE3}). By projection onto the $r$-th stable summand and taking stable adjoint, one obtains a map $j_r:CX\to QD_rX$ known as the $r$-th stable James-Hopf invariant which is ought to be an extension of the projection $F_rCX\to D_rX$ \cite[Proposition 1.1]{BE3}. By no means the map $j_r$ is a map of monoids which can be seen by homology computations as in \cite{Kuhnhomology}. If $X$ is path connected, then through the homotopy equivalence $CX\to QX$ we have maps $j_r:QX\to QD_rX$ which we continue to call the $r$-th stable James-Hopf invariant. We just note that by aid of Milnor type splitting for spaces $QCX$ \cite[Theorem A]{BE3}, that $QCX$ is an $E_\infty$ ring space in the sense of \cite{MQRT}, and that the total James-Hopf map $CX\to QCX$ enjoys an `exponential property' \cite{CCMT}, the stable James-Hopf maps admit extensions $QX\to QD_rX$ when $X$ is not necessarily path connected. Since, we have not used the maps in this case, we then omit technical details on them.
We refer the reader to \cite{Kuhngeometry} for a detailed study on geometry of the $j_r$ maps. The maps $j_r$ show interesting properties of which we recall the delooped Kahn-Priddy Theorem \cite[Corollary 2.14]{Kuhnhomology}, \cite[Corollary 6.8]{Finkelstein}(see also \cite{Kuhn-extended} for a generalisation to all primes); it asserts that there is a choice for a map $\lambda:QP\to Q_0S^0$ which induces an epimorphism in ${_2\pi_*}$-homotopy, deloops once, and its right inverse is given by $t=\Omega j_2$ with $j_2:QS^1\to QD_2S^1=Q\Sigma P$, i.e. $t\lambda$ is a ${_2\pi_*}$-equivalence. Consequently, $t$ is a monomorphism in $\Z/2$-homology.\\

Finally, we recall that the $D_r$ functor has been generalised to the category of spectra; in particular for a space $X$, there is a homotopy equivalence $D_r(\Sigma^\infty X)\to \Sigma^\infty D_rX$ \cite[Theorem A]{May-Equicons}(see also \cite{LMS}). This allows to apply $D_r$ to maps among suspension spectra. In particular, given stable maps $f:S^n\to S^0$ we obtain a stable map $D_r(f):D_rS^n\to D_rS^0$ whose stable adjoint $D_rS^n\to QD_rS^0$ extends to an infinite loop map $QD_rS^n\to QD_rS^0=Q({B\Sigma_r}_+)$. As discussed in Subsection \ref{iteratedloop} this agrees with the map $QD_rS^n\to QD_rS^0$ obtained by applying the $\Omega^\infty$ functor to the stable map $D_r(f):D_rS^n\to D_rS^0$.

\subsection{Homology of the Snaith splitting}
If $X$ is path connected, throughout the homotopy equivalence $CX\to QX$, the space $QX$ then stably splits as $\bigvee_r D_rX$ and consequently the homology decomposes as well. There is a filtration on homology of $QX$, called the height filtration (sometimes weight filtrations) $ht:H_*QX\to\mathbb{N}$ such that $ht(x_\mu)=1$ for all $x_\mu\in \widetilde{H}_*X$, $ht(\eta\xi)=ht(\xi)+ht(\eta)$, and $ht(Q^i\xi)=2ht(\xi)$; in particular $h(Q^Ix_\mu)=2^{l(I)}$. In homology we have
$$H_*D_rX=\Z\{\xi\in H_*QX:ht(\xi)=r\}$$
the $\Z$-module generated by all elements of $H_*QX$ of height $r$. As $j_r$ extends the projection $F_rCX\to D_rX$ \cite[Proposition 1.1]{BE3}, and from the stable splitting, it follows when $X$ is path connected, $(j_r)_*:H_*QX\to H_*QD_rX$ acts like projection on elements $Q^Ix_\mu$ if $l(I)=r$ and that $(j_r)_*Q^Ix_\mu=0$ if $l(I)<r$ (see also \cite{Kuhnhomology}).

\subsection{Stable adjoint of Toda brackets}\label{adjointtoda}
The material in here is ought to be well known as the stable adjoint of Toda brackets has been considered previously, e.g. \cite{MahowaldThompson}. But, we include a brief discussion for the references in this paper. The following provides natural way to represent stable adjoint of a Toda bracket of stable maps among stable complexes; we believe it generalises to higher brackets in a natural way.\\

For a moment, given a stable map $\alpha:X\to Y$, we write $\widetilde{\alpha}:X\to QY$ for its stable adjoint. Consider the following diagram
$$X\stackrel{\alpha}{\lra}Y\stackrel{\beta}{\lra}Z\stackrel{\gamma}{\lra}W$$
of stable maps and stable complexes where successive compositions are null. This allows to obtain a commutative diagram
$$\xymatrix{
X \ar[r]^-\alpha & Y\ar[r]^-\beta\ar[d]_-{\iota_\alpha}              & Z\ar[r]^-\gamma\ar[d]_-{\iota_\beta}     & W\\
                        & C_\alpha\ar[ru]^-{\overline{\beta}}\ar[d]_-p     & C_\beta\ar[ru]_-{\gamma_\sharp} \\
                        & \Sigma X\ar[ru]_-{\alpha^\flat}
                }$$
where the stable composition $\gamma_\sharp\circ\alpha^\flat$ represents the triple Toda bracket $\{\alpha,\beta,\gamma\}$ up to the indeterminacy in choosing the extension and co-extension maps in the above diagram. The construction of forming a Toda bracket, is made of taking mapping cones and finding extensions which are stable under suspension or desuspension. Therefore, we may suspend or desuspend the above diagram to obtain diagrams which represent the suspension of desuspension of the Toda bracket above, i.e. represent $\Sigma\{\alpha,\beta,\gamma\}$ or $\Sigma^{-1}\{\alpha,\beta,\gamma\}$.\\
Next, for an input data as above, we like to consider the stable adjoint of the Toda bracket $\Sigma X\to QW$ or rather its extension $Q\Sigma X\to QW$. First, note that all complexes in the above diagram, are stable complexes (suspension spectra) which allows us to take stable adjoint of everything in sight which yields
$$\xymatrix{
X \ar[r]^-{\widetilde{\alpha}}& QY\ar[r]^-{\Omega^\infty\beta}\ar[d]_-{Q\iota_\alpha}         & QZ\ar[r]^-{\Omega^\infty\gamma}\ar[d]_-{Q\iota_\beta}     & QW\\
                              & QC_\alpha\ar[ru]^-{\Omega^\infty\overline{\beta}}\ar[d]_-{Qp} & QC_\beta\ar[ru]_-{\Omega^\infty\gamma_\sharp} \\
                              & Q\Sigma X\ar[ru]_-{\Omega^\infty\alpha^\flat}
                }$$
On the other hand, by standard property of cofibre sequences, the cone of the adjoint map $C_{\widetilde{\alpha}}$ and $QC_\alpha$ are related through a map $C_{\widetilde{\alpha}}\to QC_\alpha$. This allows us to complete the above diagram to a commutative diagram
$$\xymatrix{
X \ar[rr]^-{\widetilde{\alpha}}&& QY\ar[rr]^-{\Omega^\infty\beta}\ar[dd]_-{Q\iota_\alpha}\ar[ld]&& QZ\ar[ddl]\ar[rr]^-{\Omega^\infty\gamma}\ar[dd]_-{Q\iota_\beta} && QW\\ 
&C_{\widetilde{\alpha}}\ar[dd]\ar[rd]&&&  &&\\ 
&& QC_\alpha\ar[rruu]_-{\Omega^\infty\overline{\beta}}\ar[dd]_-{Qp}&C_{\Omega^\infty\beta}\ar[r]\ar[rrruu]^-{(\Omega^\infty\gamma)_\sharp}
            &QC_\beta\ar[rruu]_-{\Omega^\infty\gamma_\sharp} \\
                                     &\Sigma X\ar[rd]\ar[rru]_{\widetilde{\alpha}^\flat}   &&&&&\\
                                     && Q\Sigma X\ar[rruu]_-{\Omega^\infty\alpha^\flat}
                }$$
which shows that the composition $\widetilde{\alpha}^\flat\circ(Q\gamma)_\sharp:\Sigma X\to QW$, which represents the stable adjoint of $\{\alpha,\beta,\gamma\}$, does factorise throughout $(Q\gamma_\sharp)\circ(Q\alpha^\flat)=Q(\gamma_\sharp\circ\alpha^\flat):Q\Sigma X\to QW$. Since the inclusion $\Sigma X\to Q\Sigma X$ induces a monomorphism in $\Z/2$-homology, hence in order to show that the stable adjoint of $\{\alpha,\beta,\gamma\}$ is trivial/nontrivial in homology, it is enough to work with $Q(\gamma_\sharp\circ\alpha^\flat)$ or rather with $(Q\gamma_\sharp)\circ\widetilde{\alpha^\flat}$ which $Q(\gamma_\sharp\circ\alpha^\flat)$ is an infinite loop extension of it. We shall use these observations and their combinations while working with homology of Toda brackets. Let's note that\\
$(1)$ by the above diagram, for the sake of homology calculations we may just work with the diagram obtained from taking adjoints;\\
$(2)$ as homology is stable under suspension, hence while working with the adjoint diagram, we will allow ourselves to sue various stable splitting results.\\
$(3)$ Similar considerations as above, are in place when we decide to work more unstably. That is, having elements an element in $\pi_i^s$ we may think of its pull back to $\pi_{2i+1}S^{i+1}$ and consider its adjoint maps $S^{2i+1-j}\to\Omega^jS^{i+1}$ and even consider the extension of this map into a $\Omega^j$-map. We then allow ourselves to form Toda bracket of such maps, which is a quite descent job to do, that will ease computations.

\section{Hurewicz homomorphism and products: Proof of Theorem \ref{main0}}
We begin with the following.

\begin{lmm}\label{ij}
Let $f\in{\pi_i}Q_0S^0,g\in{\pi_j}Q_0S^0$ with $i,j>0$ and $i\neq j$. Then $h(fg)=0$.
\end{lmm}

\begin{proof}
The product $gf$ in ${\pi_*^s}$ is determined by the composition of stable maps $S^{i+j}\stackrel{f}{\to}S^j\stackrel{g}{\to}S^0$ which we wish to compute homology of its stable adjoint $S^{i+j}\to QS^j\to QS^0$. For $i<j$ the map $S^{i+j}\lra S^j$ is in the stable range, so it can be taken as a genuine map, i.e. $S^{i+j}\to QS^j$ does factorise as $S^{i+j}\to S^j\to QS^j$ where the part $S^{i+j}\to S^j$ is not necessarily unique. The stable adjoint of $gf$ then maybe viewed as $S^{i+j}\stackrel{f}{\lra}S^j\stackrel{g}{\lra}Q_0S^0$
which is trivial in homology for dimensional reasons. The indeterminacy in choosing the pull back $S^{i+j}\to S^j$ is irrelevant here and will not effect the dimensional reason. Hence, $h(gf)=0$. The case $i>j$ is similar, noting that $\pi_*^s=\pi_*Q_0S^0$ is commutative.
\end{proof}

Note that the above result hold integrally, and consequently on ${_p\pi_*^s}$ for any prime $p$. According to the above lemma, $h(fg)\neq 0$ may occur if $f,g\in{\pi_n}Q_0S^0$. Note that for $f\in{\pi_n^s}$ there exists $\widetilde{f}\in\pi_{2n+1}S^{n+1}$, not necessarily unique, which maps to $f$ under the stabilisation $\pi_{2n+1}S^{n+1}\lra\pi_n^s$. We shall say $f$ is detected by the unstable Hopf invariant, if $\widetilde{f}$ is detected by cup-squaring operation in its mapping cone, i.e. $g_{n+1}^2\neq 0$ in $H^*(C_{\widetilde{f}};\Z)$.

\begin{thm}\label{nilpotence}
Suppose $f,g\in\pi_nQ_0S^0$  with $h(fg)\neq 0$ with $h:\pi_*Q_0S^0\lra H_*(Q_0S^0;\Z)$ being the Hurewicz homomorphism. Then both $f$ and $g$ are detected by the unstable Hopf invariant.
\end{thm}

\begin{proof}
We are interested in the stable adjoint of $S^{2n}\stackrel{f}{\lra}S^n\stackrel{g}{\lra}S^0$ given by
$$S^{2n}\stackrel{f_n}{\lra} QS^n\stackrel{g}{\lra} Q_0S^0.$$
As noted above, for dimensional reasons, the mapping $f_n$ factors as $S^{2n+1}\stackrel{\widetilde{f}_n}{\lra}\Omega S^{n+1}\lra QS^n$ where $\widetilde{f}_n$, which is not necessarily unique, is the adjoint for an appropriate $\widetilde{f}$; the map $i:\Omega S^{n+1}\lra QS^n$ is the stablisation map. Hence, the stable adjoint of $fg$ can be seen as a composite
$$S^{2n}\lra \Omega S^{n+1}\lra QS^n\lra Q_0S^0.$$
Now, $h(fg)\neq 0$ implies that $h(\widetilde{f}_n)\neq 0$. This shows that $h(\widetilde{f}_n)=\lambda g_n^2$ for some nonzero $\lambda\in\Z$. This latter implies that $h(f_n)=\lambda g_n^2$. On the other hand, it is well know that $h(\widetilde{f}_n)=\lambda g_n^2$ if and only if $g_{n+1}^2=\pm\lambda g_{2n+2}$ in $H^*(C_{\widetilde{f}};\Z)$, i.e. $f$ is detected by the unstable Hopf invariant (see for example \cite[Proposition 6.1.5]{Harper}). Similarly, $g$ is also detected by the unstable Hopf invariant.
\end{proof}

The result in integral case, implies the $p$-primary case. In particular, we have the following.

\begin{crl}
Suppose $f,g\in{_2\pi_n}Q_0S^0$  with $h(fg)\neq 0$ then both $f$ and $g$ are Hopf invariant one elements, i.e. $f,g=\eta,\nu,\sigma$. Here,  $h:{_2\pi_*}Q_0S^0\lra H_*(Q_0S^0;\Z/2)$ is the mod $2$ Hurewicz homomorphism.
\end{crl}

In order to complete the proof of Theorem \ref{main0} note that by above observations the nontrivial image of the Hurewicz homomorphism on the the ideal $\la f:f=\eta,\nu,\sigma\ra$ can only arise from elements $f,f^2$. Moreover, choosing $g:S^0\lra S^0$ to be a stable map of odd degree, $f=\eta,\nu,\sigma$, then $fg$ is an odd multiple of a Hopf invariant one element which we know map nontrivially under $h$.\\
For the case of $h^\circ$ we note that given $\alpha\in\pi_iQ_0S^0$ and $f\in\pi_j\Omega^\infty E$ the Hurewicz image of $f\circ \alpha$ is determined by the homology of the composition
$$S^{i+j}\stackrel{\alpha}{\lra} QS^j\stackrel{f}{\lra} \Omega^\infty E.$$
By similar reasons, if $i<j$ then by Fruedenthal's theorem $\alpha$ doesfactorise through some map $S^{i+j}\to S^j$, hence $h^\circ(\alpha\circ f)=0$. If $i=j$ then $\alpha$ admits a factorisation as $S^{2i}\to\Omega\Sigma S^i\to QS^j\to \Omega^\infty E$ which if is nontrivial in homology which imply that $\alpha$ is detected by the unstable Hopf invariant, so $\alpha\in\{\eta,\nu,\sigma\}$. Finally, for the case $i>j$, note that $(i-j+1)$-th adjoint of $\alpha:S^{i+j}\to QS^j$ is a map $S^{2i+1}\to QS^{i+1}$ which does factorise throughout $S^{2i+1}\to S^{i+1}$ so $e_*^{i-j+1}h^\circ(\alpha\circ f)=0$.


\section{Proof of Theorem \ref{ideal}}
Suppose $S\subseteq{_2\pi_{*>0}^s}$ on which the Curtis conjecture holds, i.e. if $f\in S$ and $h(f)\neq 0$ then $f$ equals to a Hopf invariant one or Kervaire invariant one element, or an odd multiple of them, modulo other terms which all vanish under $h$. We wish to verify that the conjecture holds on $\la S\ra$.\\

Since ${_2\pi_*^s}$ is a graded ring then a typical element in $\la S\ra$ might be written as $\int \alpha f$ with $f\in S$ and $\alpha\in{_2\pi_j^s}$ where we may arrange terms of this sum according to their grading. Since $h$ is a $\Z$-module homomorphism, it is then enough to verify the theorem at each grading. Consider a finite sum $\int_{i+j=n,i>0}\alpha_jf_i$ with $f_i\in {_2\pi_i^s}\cap S$ and $\alpha_j\in{_2\pi_j^s}$ such that $h(\int_{i+j=n,i>0} \alpha_jf_i)\neq 0$. Then, there is at least one term $\alpha_jf_i$  with $h(\alpha_jf_i)\neq 0$. We have two cases: $j=0$ and $j>0$.\\

\tb{Case} $j>0$. By Theorem \ref{main0}, $f_i=\alpha_j$ is a Hopf invariant one element or an odd multiple of these elements. Therefore, $\alpha_jf_i$ is one of $\eta^2,\nu^2,\sigma^2$, living in gradings $2,6,14$ respectively. Note that $2$-component of $2$-stem and $6$-stem are known to be isomorphic to $\Z/2$. Therefore, in these case the sum $\int \alpha_jf_i$ has only one term which we determined above to be one of the Kervaire invariant one elements $\eta^2,\nu^2$. This proves the theorem for this case. If $\alpha_jf_i=\sigma^2$ then the sum $\int \alpha_jf_i$ lives in ${_2\pi_{14}^s}\simeq(\Z/2)^{\oplus 2}$ generated respectively by $\sigma^2$ and $\kappa$. Consequently, the sum $\int \alpha_jf_i$ can be written as $\sigma^2+\alpha\kappa$ with $\alpha\in{_2\pi_0^s}$ which could be a map of even or odd degree. In any case, as $h(\kappa)=0$, the result follows.\\

\tb{Case} $j=0$. If $j=0$ then $\alpha_j$ is a map of odd degree which means $h(f_i)=h(\alpha_jf_i)\neq 0$. By assumption $f$ is either a Hopf invariant one element or Kervaire invariant one element or an odd multiple of these elements, so $\alpha_jf_i$ is. If $\alpha f$ is a Hopf invariant one element, then the sum $\int \alpha_jf_i$ lives in one of $1$-, $3$-, or $7$-stems which are generated by $\eta,\nu,\sigma$ respectively. Hence, the whole sum $\int \alpha_jf_i$ must be an odd multiple of these elements. This verifies the theorem in this case. The case of $f$ being a Kervaire invariant one is verified similar to the above case.\\

\subsection{Computing $h(\kappa)$ and $h(\overline{\nu})$}

We begin with $\overline{\nu}\in{_2\pi_8^s}$ which is an element of order $2$.
\begin{lmm}\label{h(nu-bar)}
Let $\overline{\nu}\in{_2\pi_8^s}$ be a generator represented by the Toda bracekt $\{\nu,\eta,\nu\}$. Then it maps trivially under $h:{_2\pi_8}Q_0S^0\to H_8(Q_0S^0;\Z/2)$.
\end{lmm}

\begin{proof}
Consider the composition $S^7\stackrel{\nu}{\lra} S^4\stackrel{\eta}{\lra} S^3\stackrel{\nu}{\lra} Q_0S^0$ where all maps are genuine maps of spaces and  successive compositions are trivial. This yields extension and co-extension maps $\nu^\flat,\nu_\sharp$ that we can compose as
$S^8\stackrel{\nu^\flat}{\lra}C_\eta\stackrel{\nu_\sharp}{\lra} Q_0S^0$
which represents $\overline{\nu}:S^8\to Q_0S^0$. Obviously, $(\nu^\flat)_*=0$ as $C_\eta$ has its top cell in dimension $5$. Hence, $h(\overline{\nu})=0$. Note that the vanishing happens for dimensional reasons and the indeterminacy in choosing the extension and co-extension maps is irrelevant here.
\end{proof}

Next consider the generator $\kappa\in {_2\pi_{14}^s}$ which is an element of order $2$. According to Toda \cite{Toda}, it can be represented by a triple Toda bracket $\{\beta,\alpha,\nu\}$ (our order of maps in the bracket is opposite to Toda, and we have dropped the suspensions from our notation for simplicity). Here, $\alpha:\Sigma K\to S^0$ is an extension of $\eta$ implied by $2\eta=0$ with $K$ being the stable complex $S^0\cup_2e^1$, and $\beta:S^9\to K$ is a coextension of $\overline{\nu}\in{_2\pi_8^s}$. Note that $\Sigma K$ is just $\R P^2$.

\begin{lmm}
By abuse of notation, write $\alpha:\Sigma^3(\Sigma K)\to S^3$ for the third suspension of $\alpha:\Sigma K\to S^0$. Then, for the mapping cone of $\Gamma^6\alpha:\Gamma^6(\Sigma^4K)\to\Gamma^6S^3$ we have
$$H_*(C_{\Gamma^6\alpha};\Z/2)\simeq H_*(\Gamma^6 S^3;\Z/2)\oplus H_{*+1}(\Gamma^6(\Sigma^4K);\Z/2)$$
where $\Gamma^n=\Omega^n\Sigma^n$. Write $g_3\in H_3(\Gamma^6S^3;\Z/2)$ for a generator coming from $S^3$, and $\Sigma^3a_1,\Sigma^3a_2\in H_*(\Gamma^6(\Sigma^4K);\Z/2)$ for the $4$- and $5$-dimensional generators coming from $\Sigma K=\R P^2$ with the relation $Sq^1_*a_2=a_1$. The action of the Steenrod operations $Sq^t_*$ on $H_*(C_{\Gamma^6\alpha};\Z/2)$ then is determined with its action on $\Gamma^6S^3$, $\Gamma^6(\Sigma^4K)$ and the additional relation
$$Sq^2_*\Sigma_*(\Sigma^3a_1)=g_3$$
together with Nishida relations where $\Sigma_*(\Sigma^3a_1)\in H_5(C_{\Gamma^6\alpha};\Z/2)$ is a generator projecting onto a generator of $H_5(\Sigma\Gamma^6(\Sigma^4K);\Z/2)$ under the pinch map $C_{\Gamma^6\alpha}\to\Sigma\Gamma^6(\Sigma^4K)$.
\end{lmm}

\begin{proof}
The mapping $\alpha:\Sigma^3(\Sigma K)\to S^3$ is in the stable range whose homology is trivial for dimensional reasons. The homology of $\Gamma^6\alpha$ is determined by $(\Gamma^6\alpha_*)Q^Ix=Q^I\alpha_*x$ where $x\in H_*\Sigma^4K$ is a generator. Consequently, $(\Gamma^6\alpha)_*=0$. Therefore, the homology decomposes as claimed. For the action of the Steenrod operations, we only need to justify the additional relation and the others are standard. But, this comes form the fact that $\alpha$ is an extension of $\eta$ and there is nontrivial $Sq^2$ action in $\eta$, and that, at least stably, there is commutative diagram
$$\xymatrix{\Gamma^6(\Sigma^4K)\ar[r]^-{\Gamma^6\alpha} & \Gamma^6S^3\\
            \Sigma^4K\ar[u]\ar[r]^-{\alpha}             & S^3.\ar[u]}$$
Since, Steenrod operations are stable this then completes the proof.
\end{proof}

\begin{lmm}\label{h(kappa)}
Let $\kappa\in{_2\pi_{14}^s}$ be a generator of order $2$ given by the a triple Toda bracket $\{\beta,\alpha,\nu\}$. Then $h(\kappa)=0$ for the Hurewicz homomorphism $h:{_2\pi_{14}}Q_0S^0\to H_{14}(Q_0S^0;\Z/2)$.
\end{lmm}

\begin{proof}
The Toda bracket representing $\kappa$ is obtained from considering $S^{13}\to\Sigma^4K\to S^3\to S^0$ which yields extension and co-extension maps that fit together to represent $\kappa$ as $S^{14}\to C_\alpha\to S^0$. In order to work in the level of spaces, we need to realise the stable maps as maps of spaces. We proceed as follows. We think of $\nu$ as a genuine map $S^3\lra Q_0S^0$. Also note that $\Sigma^4K\to S^3$ is already in the stable range, so it can be taken as a map of spaces. For $\beta:S^9\to K$, as $K$ has its bottom cell in dimension $0$, we need to suspend $10$ times which yields $S^{19}\to\Sigma^{10}K$. We then think of $\beta$ in terms of its $6$-th adjoint as $S^{13}\to\Omega^6\Sigma^{6}(\Sigma^4K)$. Now, consider the composition
$$S^{13}\stackrel{\beta}{\lra}\Gamma^6\Sigma^4K\stackrel{\Gamma^6\alpha}{\lra}\Gamma^6 S^3\lra Q_0S^0$$
where $\Gamma^6=\Omega^6\Sigma^6$ and the map $\Gamma^6S^3\to Q_0S^0$ is the extension of $S^3\to Q_0S^0$ obtained from the fact that $Q_0S^0$ is already a $6$-fold loop spaces. The successive compositions are trivial. This provides extension and co-extension maps that we may compose as $$S^{14}\stackrel{\beta^\flat}{\lra}C_{\Gamma^6\alpha}\stackrel{\nu_\sharp}{\lra}Q_0S^0$$
which represents $\kappa$ as an element in ${_2\pi_{14}}Q_0S^0$. We claim $\beta^\flat_*=0$, that is there is no spherical class in $H_{14}C_{\Gamma^6\alpha}$. If $h(\beta^\flat)\neq 0$ then $Sq^t_*h(\beta^\flat)=0$ for all $t>0$. Applying the above lemma, we compute that $H_{14}(C_{\Gamma^6\alpha};\Z/2)$ has generators
$$g_3^2(Q^5g_3)=(Q^3g_3)(Q^5g_3),\ g_3Q^8g_3,\ \Sigma_*((\Sigma^3a_1)^2(\Sigma^3a_2)),\ \Sigma_*((Q^5\Sigma^3a_1)(\Sigma^3a_1)),\ \Sigma_*Q^8\Sigma^3a_2.$$
Hence, if $h(\beta^\flat)\neq 0$ then it could be written as
$$\epsilon_1(Q^3g_3)(Q^5g_3)+\epsilon_2g_3Q^8g_3
               +\epsilon_3\Sigma_*((\Sigma^3a_1)^2(\Sigma^3a_2))+\epsilon_4\Sigma_*((Q^5\Sigma^3a_1)(\Sigma^3a_1))+\epsilon_5\Sigma_*Q^8\Sigma^3a_2$$
with $\epsilon_i\in\Z/2$ and at least one of then is nonzero. First, consider the component coming from $H_*\Gamma^6S^3$. Since,
$$Sq^2_*(Q^3g_3)(Q^5g_3)=(Q^3g_3)^2=g_3^4,\ Sq^2_*g_3Q^8g_3=g_3Q^6g_3$$
we conclude that it is impossible to have $\epsilon_1+\epsilon_2=1$ in $\Z/2$ as if one of $\epsilon_1$ or $\epsilon_2$ is nontrivial, then the component of $h(\beta^\flat)$ in $H_*\Gamma^6S^3$ is not annihilated by $Sq^2_*$, so $h(\beta^\flat)$ is not annihilated by $Sq^2_*$ which is a contradiction. Moreover, if $\epsilon_1=\epsilon_2=1$ then
$$Sq^2_*((Q^3g_3)(Q^5g_3)+g_3Q^8g_3)=g_3^4+g_3Q^6g_3\neq 0.$$
This shows that if $h(\beta^\flat)\neq 0$, then it cannot have any nontrivial component in $H_*\Gamma^6S^3$. Therefore, if $h(\beta^\flat)\neq0$ then it has to project nontrivially onto $H_*(\Sigma\Gamma^6(\Sigma^4K);\Z/2)$ under the pinch map, i.e.
$$h(p\circ\beta^\flat)=\epsilon_3\Sigma_*((\Sigma^3a_1)^2(\Sigma^3a_2))+\epsilon_4\Sigma_*((Q^5\Sigma^3a_1)(\Sigma^3a_1))+\epsilon_5\Sigma_*Q^8\Sigma^3a_2\neq 0$$
where $p:C_{\Gamma^6\alpha}\to\Sigma\Gamma^6(\Sigma^4K)$ is the pinch map. On the other hand, by the construction of the Toda bracekt, $p\circ\beta^\flat=\Sigma\beta$, hence
$$h(\Sigma\beta)=\epsilon_3\Sigma_*((\Sigma^3a_1)^2(\Sigma^3a_2))+\epsilon_4\Sigma_*((Q^5\Sigma^3a_1)(\Sigma^3a_1))+\epsilon_5\Sigma_*Q^8\Sigma^3a_2$$
which implies that
$$h(\beta)=\epsilon_3(\Sigma^3a_1)^2(\Sigma^3a_2)+\epsilon_4(Q^5\Sigma^3a_1)(\Sigma^3a_1)+\epsilon_5Q^8\Sigma^3a_2\in H_{13}(\Gamma^6(\Sigma^4K);\Z/2)$$
has to be nontrivial. Note that the first two terms in the above sum, if nontrivial, are decomposable classes, so by Cartan formulra for $Sq^t_*$ operations, map to decomposable class under $Sq^t_*$. Now, if $\epsilon_5\neq 0$ then we compute that
$$Sq^1_*h(\beta)=Q^7\Sigma^3a_2+\textrm{ decomposable terms}\neq 0$$
which is a contradiction, hence $\epsilon_5=0$. Therefore, $h(\beta)$ consists of decomposable classes in the Hopf algebra $H_*(\Gamma^6(\Sigma^4K);\Z/2)$. Since $h(\beta)$ is also primitive, hence by \cite[Proposition 4.21]{MM} it has to be a square. But, it is obvious non of the first two terms in $h(\beta)$ neither their sum is a square. This is a contradiction as it shows that $h(\beta)=0$ which contradicts the requirement that $h(p\circ\beta^\flat)=h(\Sigma\beta)\neq 0$. Therefore, $h(\beta^\flat)=0$. Moreover, our computations do not depend on a specific choice of $\beta^\flat$, hence the indeterminacy in choosing $\beta^\flat$ is irrelevant here, therefore, $h(\kappa)=0$.
\end{proof}

Let us note that $\kappa$ can also be represented by a higher Toda bracket which consists only of the Hopf invariant one maps and $2$, namely $\{\nu,\overline{\nu},2,\eta\}=\{\nu,\{\nu,\eta,\nu\},2,\eta\}$. The above proof then shows that the Hurewicz image of the element represented by this Toda bracket is trivial, a direct proof is however more delicate and involves too much detailed computations.

\section{Proof of Lemma \ref{main1}}
By \cite[Theorem 1]{Za} Curtis conjecture holds on ${_2\pi_*}J$. Assuming that the conjecture also holds on $F^2$, and that $h(\tau_i)=0$, the theorem follows upon applying Theorem \ref{ideal}. Note that the Adams filtration of the elements in ${_2\pi_*}J$, consequently the elements of $\la{_2\pi_*}J\ra$, is not bounded \cite[Theorem 4.2]{Ma1}. Since the Theorem holds on $\la{_2\pi_*}J\ra$ then we see that the Adams filtration of the elements on which the conjecture hold is not bounded.\\
Hence, in order to complete the proof we have to verify the conjecture on $F^2$; we postpone computation of $h(\tau_i)$ to another section. The elements of $F^2$ are known: the Hopf invariant one elements $\eta,\nu,\sigma$; the Kervaire invariant one elements $\theta_i$ with $i\leqslant 6$; elements of Mahowald's family $\mu_i$; and the element $\nu^*\in{_2\pi_{18}^s}$. Note that in the tables of \cite[Appendix 3]{Ravenel-Greenbook} there is also the element $\eta^*\in{_2\pi_{16}^s}$ represented by $\{\sigma,2\sigma,\eta\}$; this is just the element $\eta_4$ of Mahowald's family the is been considered below. The conjecture then is verified on $F^2$ if we show that $h(\mu_i)=0$ and $h(\nu^*)=0$. First we deal with $h(\mu_i)$ which is straightforward.

\begin{lmm}\label{eta}
Let $\eta_i\in{_2\pi_{2^i}}Q_0S^0$, $i\geqslant 3$, denote Maholwad's family. Then $h(\eta_i)=0$ where
$h:{_2\pi_{2^i}}Q_0S^0\lra H_{2^i}Q_0S^0$ is the Hurewicz homomorphism.
\end{lmm}
\begin{proof}
Mahowald's $\eta_i$ family, living in ${_2\pi_{2^i}^s}$ for $i\geqslant 3$, is constructed as a composite of maps \cite{Ma}
$S^{2^i}\stackrel{f_i}{\longrightarrow} X_i\stackrel{g_i}{\longrightarrow} S^0$ where $X_i:=F(\R^2,2^i-3)\ltimes_{\Sigma_{2^i-3}}(S^7)^{\wedge(2^i-3)}$, with $F(\R^2,2^i-3)$ being the configuration space of $2^i-3$ distinct points in $\R^2$, has its bottom cell is in dimension $2^i-2^{i-3}=2^{i-3}7$ and top cell in a dimension less than $2^i$. Therefore, the complex $X_i$ is a genuine complex and the mapping $f_i$ can be taken as a genuine mapping. The stable adjoint of $\eta_i$ then is given by
$S^{2^i}\stackrel{f_i}{\longrightarrow} X_i\stackrel{g_i}{\longrightarrow} Q_0S^0$
where ${f_i}_*=0$ for dimensional reasons. Hence, ${\eta_i}_*=0$.
\end{proof}

\subsection{Computing $h(\nu^*)$}

We begin with a computational result.

\begin{lmm}\label{sigmabar}
Consider $\sigma,2\sigma\in{_2\pi_7^s}$ and let $\overline{\sigma}:C_{\sigma}\lra QS^3$ be an extension of $2\sigma$ implied by $(2\sigma)\circ\sigma=0$. Then $\overline{\sigma}_*=0$.
\end{lmm}

\begin{proof}
We want to show that for $\sigma:S^{17}\lra S^{10}$ and $2\sigma:S^{10}\lra QS^3$, then $\overline{\sigma}_*=0$. Since $\overline{\sigma}$ extends $2\sigma$, hence its homology on the bottom cell is determined by $2\sigma$ which is trivial in $\Z/2$-homology. For the top cell, $\overline{\sigma}_*g_{18}\in H_{18}QS^3$ which we compute as follows. First, note that the map $\sigma:S^{17}\to S^{10}$ desuspend, at least once, to a map $\sigma_{-1}:S^{16}\to S^9$, hence $C_\sigma\simeq\Sigma C_{\sigma_{-1}}$. Taking adjoint of the (null) composition $S^{17}\to S^{10}\to QS^3$ we obtain $S^{16}\to S^9\to QS^2$ which provides a map $\overline{\sigma}_{-1}:C_{\sigma_{-1}}\to QS^2$ such that
$$
\xymatrix{
\Sigma C_{\sigma_{-1}}\ar[r]^-{\Sigma\overline{\sigma}_{-1}}\ar[d]_-{\simeq} & \Sigma QS^2\ar[d]^-{e}\\
C_\sigma\ar[r]^-{\overline{\sigma}}                                          & QS^3}
$$
commutes, where $e$ is the evaluation map. This implies that $\overline{\sigma}_*g_{18}=e_*{\overline{\sigma}_{-1}}_*g_{17}$. Note that $e_*$ is just homology suspension $H_*QS^2\to H_{*+1}QS^3$. Note that $C_{\sigma_{-1}}$ also desuspends, at least once. Hence, $H^*C_{\sigma_{-1}}$ has trivial ring structure and $H_*C_{\sigma_{-1}}$ is primitively generated; in particular $g_{17}$ is primitive, so ${\overline{\sigma}_{-1}}_*g_{17}\in H_{17}QS^2$ is primitive. Now, write $${\overline{\sigma}_{-1}}_*g_{17}=\int Q^Ig_{2}+D$$
where the sum could be empty and $D$ denotes a sum of decomposable terms. Since each term $Q^Ig_{17}$ is primitive, hence $D$ is also a primitive class which by \cite[Proposition 4.23]{MM}, working at $p=2$, $D$ has to be square term but lives in odd dimensions, so $D=0$. Therefore, we may write
$$\overline{\sigma}_*g_{18}=e_*(\int Q^Ig_{2})=\int Q^Ig_3.$$
The only nonzero admissible monomials $Q^Ig_3$ in $H_{18}QS^3$ are
$$Q^{15}g_3,Q^{10}Q^5g_3,Q^9Q^6g_3=(Q^6g_3)^2.$$
Hence, if $\overline{\sigma}_*g_{18}\neq 0$ then
$$\overline{\sigma}_*g_{18}=\epsilon_1Q^{15}g_3+\epsilon_2Q^{10}Q^5g_3+\epsilon_3Q^9Q^6g_3$$
with $\epsilon_i\in\Z/2$ such that at least one of them is nonzero. By Nishida relations $Sq^1_*Q^{2t}=Q^{2t-1}$ and $Sq^1_*Q^{2t+1}=0$ we compute that
$$Sq^1_*\overline{\sigma}_*g_{18}=\epsilon_2Q^{9}Q^5g_3$$
which, if $\epsilon_2=1$, then it contradicts the fact that $Sq^1_*\overline{\sigma}_*g_{18}=\overline{\sigma}_*Sq^1_*g_{18}=0$ as $C_{\sigma}$ has cells only in dimensions $10$ and $18$. Hence, $\overline{\sigma}_*g_{18}=\epsilon_1Q^{15}g_3+\epsilon_3Q^9Q^6g_3$. Now, applying $Sq^2_*$, we obtain
$$Sq^2_*\overline{\sigma}_*g_{18}=\epsilon_3Q^8Q^5g_3=\epsilon_3(Q^5g_3)^2.$$
Hence, if $\epsilon_3=1$ then we run into contradiction as by a similar reasoning, since $C_\sigma$ has no cells in dimension $16$ we must have $Sq^2_*\overline{\sigma}_*g_{18}=0$. Consequently, the only remaining possibility is
$$\overline{\sigma}_*g_{18}=Q^{15}g_3.$$
In order to eliminate this last possibility, we proceed with some unstable computations. The element $2\sigma\in\pi_7^s\simeq\pi_{10}QS^3$ pulls back to $\pi_{15}S^8$, hence $2\sigma:S^{10}\lra QS^3$ pulls back to a map $S^{10}\lra\Omega^5S^8$. The composition $S^{17}\lra S^{10}\lra \Omega^5S^8$ is trivial which provides a map $\overline{\sigma}_1:C_\sigma\lra \Omega^5S^8$, providing a factorisation of $\overline{\sigma}$ as $C_\sigma\lra \Omega^5S^8\lra QS^3$. Hence, if $\overline{\sigma}_*g_{18}\neq 0$ implies that $(\overline{\sigma_1})_*g_{18}\neq 0$. Now, we may list the elements of $H_{18}(\Omega^5S^8;\Z/2)$ which are
$$Q^{10}Q^5g_3,Q^9Q^6g_3=(Q^6g_3)^2$$
and can contribute to terms in $(\overline{\sigma_1})_*g_{18}$. But, these are eliminated as similar to above. Moreover, the class $Q^{15}g_3=Q_{12}g_3$ cannot live in $H_*(\Omega^5S^8;\Z/2)$ as we need at least 13 loops to have $Q_{12}$. Therefore, we cannot have $\overline{\sigma}_*g_{18}\neq 0$. This completes the proof.
\end{proof}

Now, we may complete computation of $h(\nu^*)$.

\begin{thm}\label{nu^*}
Let $\nu^*\in{_2\pi_{18}^s}$ be the element determined by Toda bracket $\{\sigma,2\sigma,\nu\}$. Then it maps trivially under $h:{_2\pi_{18}}Q_0S^0\lra H_{18}(Q_0S^0;\Z/2)$.
\end{thm}

\begin{proof}
Consider the stable adjoint of $2\sigma$ and the infinite loop extension of the stable adjoint of $\nu$ as
$$2\sigma:S^{10}\lra QS^3,\ \nu:QS^3\lra Q_0S^0.$$
For the composition $S^{17}\stackrel{\sigma}{\lra} S^{10}\stackrel{2\sigma}{\lra} S^3\stackrel{\nu}{\lra} S^0$, the map $S^{17}\lra S^{10}$ is in the stable range, hence the stable adjoint of this composition is the top row in the following diagram
$$\xymatrix{
S^{17}\ar[r]^-{\sigma} &  S^{10}\ar[r]^-{2\sigma}\ar[d]                   & QS^3\ar[r]^-{\nu}\ar[d] & Q_0S^0\\
                       &  C_{\sigma}\ar[ru]_-{\overline{\sigma}}\ar[d]_-p & C_{2\sigma}\ar[ru]_-{\nu_\natural}\\
                       &  S^{18}\ar[ru]_-{\sigma^\flat}}$$
where the successive compositions $(2\sigma)\circ\sigma$ and $\nu\circ(2\sigma)$ are trivial, and give rise to the extension and coextension maps $\sigma^\flat,\nu_\natural$ whose composition is ought to realise $\nu^*$, i.e.
$$S^{18}\stackrel{\sigma^\flat}{\lra}C_{2\sigma}\stackrel{\nu_\natural}{\lra}Q_0S^0$$
represents the Toda bracket $\{\sigma,2\sigma,\nu\}$ for $\nu^*$. Since the map $\sigma:S^{17}\to S^{10}$ has trivial homology, therefore the generator $g_{18}\in H_{18}S^{18}$ is in the image of the pinch map $C_{\sigma}\to S^{18}$. Therefore, by the previous lemma we have
$$(\nu_\natural\circ\sigma^\flat)_*g_{18}=(\nu_\natural\circ\sigma^\flat\circ p)_*g_{18}=\nu_*\overline{\sigma}_*g_{18}=0.$$
We do not rely on specific choices of $\sigma^\flat$ or $\nu_\natural$. Therefore, $h(\nu^*)=0$. This completes the proof.
\end{proof}



\section{Hurewicz homomorphism and Toda brackets}\label{Todasection}
Examples such as computing $h(\kappa)$ and $h(\overline{\nu})$ show that computing Hurewicz image of an element represented by a Toda bracket could be very tedious and involved. However, in some cases, it is possible to use dimensional arguments. We provide two partial results in direction. We first prove an integral version.
\begin{thm}\label{Todavanish}
(i) Let $\alpha\in\pi_i^s$, $\beta\in\pi_j^s$, and $\gamma\in\pi_k^s$ with $i<j+k$ and $j<k$ so that the Toda bracket $\{\alpha,\beta,\gamma\}$ is defined. Then the element of $\pi_{i+j+k+1}^s$ represented by Toda bracket $\{\alpha,\beta,\gamma\}$ maps trivially under the Hurewicz homomorphism $h:\pi_*Q_0S^0\to H_*(Q_0S^0;\Z)$.\\
(ii) Suppose $\alpha\in\pi_i^sX_1$. Let $\beta:X_1\to X_2$, $\gamma:X_2\to S^0$ be stable maps, with $X_1,X_2$ being (stable) complexes, $i<2\mathrm{conn}(X_1)$, $\dim X_1<2\mathrm{conn}(X_2)$, $i>\dim X_1>\dim X_2$, so that the Toda bracket $\{\alpha,\beta,\gamma\}$ is defined. Here, $\mathrm{conn}(X)=n$ if $\pi_t^sX\simeq 0$ for all $t<n$. Then, the element of $\pi_{i+1}Q_0S^0$ represented by Toda bracket $\{\alpha,\beta,\gamma\}$ maps trivially under the Hurewicz homomorphism $h:\pi_{i+1}Q_0S^0\to H_{i+1}(Q_0S^0;\Z)$.
\end{thm}

\begin{proof}
First, we prove (i). The conditions $i<j+k$ and $j<k$ ensure that the maps $\alpha:S^{i+j+k}\to S^{j+k}$ and $\beta:S^{j+k}\to S^k$ are in the stale range. Hence, in order to realise the Toda bracket $\{\alpha,\beta,\gamma\}$ as an element in $\pi_*Q_0S^0$ it is enough to consider the composition
$$S^{i+j+k}\lra S^{j+k}\lra S^k\lra QS^0$$
where the successive compositions are trivial. This leads to extension and co-extension maps $\alpha^\flat:S^{i+j+k+1}\to C_\beta$ and $\gamma_\sharp:C_\beta\to QS^0$ with the composition $S^{i+j+k+1}\to C_\beta\to QS^0$ representing $\{\alpha,\beta,\gamma\}$ in $\pi_*QS^0$ up to indeterminacy in choosing the extension and coextension maps. Notice that the space $C_\beta$ is a genuine complex with its top cell in dimension $j+k+1$. Hence, regardless the choice of $\alpha^\flat$, for dimensional reasons $\alpha^\flat_*=0$, hence $(\gamma_\sharp\circ\alpha^\flat)_*=0$. This implies that the Toda bracket $\{\alpha,\beta,\gamma\}$ represents an element which acts trivially in homology, i.e. it maps trivially under the Hurewicz homomorphism $h:\pi_*Q_0S^0\to H_*(Q_0S^0;\Z)$.\\
The proof of (ii) is similar. In order to realise $\{\alpha,\beta,\gamma\}$ as an element of $\pi_*Q_0S^0$, think of $\gamma$ as $\gamma:X_2\to Q_0S^0$. The conditions on the connectivity and dimension of the complexes ensure that the maps $\alpha,\beta$ are genuine maps among complexes. We then may consider the composite
$$S^i\lra X_1\lra X_2\lra Q_0S^0$$
where the successive compositions are trivial. This yields extension and coextension maps $\alpha^\flat:S^{i+1}\to C_\beta$ and $\gamma_\sharp:C_\beta\to Q_0S^0$ with the composition $S^{i+1}\to C_\beta\to Q_0S^0$ realising $\{\alpha,\beta,\gamma\}$ up to indeterminacy in choosing extension and coextension maps. However, the condition $\dim X_1>\dim X_2$ implies that, regardless the indeterminacy, $C_\beta$ has its top cell in dimension $\dim X_1+1<i+1$. Therefore, $\alpha^\flat_*=0$ and consequently $(\gamma_\sharp\circ\alpha^\flat)_*=0$. This completes the proof.
\end{proof}

The integral case, implies the $p$-primary case. Next, we provide some necessary condition for nonvanishing of the Hurewicz image of an elements in ${_2\pi_*}Q_0S^0$ which is represented by a Toda bracket. We have the following.

\begin{prp}\label{Toda2}
Let $\alpha\in{_2\pi_i^s}X_1$, $\beta:X_1\to X_2$ and $\gamma:X_2\to S^0$ be stable maps with $X_1$ and $X_2$ being stable complexes of finite type. Suppose that the Toda bracket $\{\alpha,\beta,\gamma\}$ is defined and represents $f\in{_2\pi_*^s}$. Then the following statements are equivalent.\\
(i) If $\alpha:S^i\to QX_1$ is trivial in homology and $f$ maps nontrivially under $h:{_2\pi_*^s}\simeq{_2\pi_*}Q_0S^0\to H_*(Q_0S^0;\Z/2)$ then $\gamma:QX_2\to QS^0$ is nontrivial in homology.\\
(ii) If $\gamma:S^i\to QX_1$ is trivial in homology and $f$ maps nontrivially under $h:{_2\pi_*^s}\simeq{_2\pi_*}Q_0S^0\to H_*(Q_0S^0;\Z/2)$ then $\alpha:QX_2\to QS^0$ is nontrivial in homology.\\
(iii) If $\alpha:S^i\to QX_1$ and $\gamma:QX_2\to QS^0$ are trivial in homology, then $f$ maps trivially under the Hurewicz homomorphism $h:{_2\pi_*^s}\simeq{_2\pi_*}Q_0S^0\to H_*(Q_0S^0;\Z/2)$.\\
(iv) If $h(f)\neq 0$ then either $\alpha_*\neq 0$ or $\gamma_*\neq 0$.
\end{prp}

Note that the statements (i) to (iv) are logically equivalent, and we have only mentioned them for the sake of completeness.

\begin{proof}
We prove (i). Since the Toda bracket is defined, we may build a commutative diagram as
$$\xymatrix{
S^{i}\ar[r]^-\alpha & QX_1\ar[r]^-\beta\ar[d]_-{\iota_\alpha}              & QX_2\ar[r]^\gamma\ar[d]_-{\iota_\beta}     & QS^0\\
                        & C_\alpha\ar[ru]^-{\overline{\beta}}\ar[d]_-p     & C_\beta\ar[ru]_-{\gamma_\sharp} \\
                        & S^{i+1}\ar[ru]_-{\alpha^\flat}
                }$$
where the diagonal maps are the extensions implied by $\beta\alpha=0$ and $\gamma\beta=0$. The composition $\gamma_\sharp\circ\alpha^\flat$ represents $f$ up the indeterminacy in choosing the extension and co-extension maps. The assumption $h(f)\neq 0$, however, will imply that with any choice for $\alpha^\flat$ and $\gamma_\sharp$ we must have $(\gamma_\sharp\circ\alpha^\flat)_*\neq 0$ which implies $(\gamma_\sharp)_*\neq 0$. Since $\alpha_*=0$ then the generator $g\in H_{i+j+k}S^{i+j+k}$ belongs to the image of $p_*$. The commutativity of the diagram then implies that $(\gamma_\sharp\circ\iota_\beta\circ\overline{\beta})_*\neq0$ which shows that $\gamma_*=(\gamma_\sharp\circ\iota_\beta)_*\neq 0$.
\end{proof}

\begin{note}
Proposition \ref{Toda2} leaves us with the cases $\alpha_*\neq 0$. The assumption $h(f)\neq 0$ implies that $\alpha^\flat_*\neq 0$. Since $(\Sigma\alpha)_*\neq 0$ then $q:C_\beta\to\Sigma QS^{j+k}$ is nontrivial in homology and in particular $q_*\alpha^\flat_*(g)\neq 0$. This latter means that $\alpha^\flat_*(g)\in H_*C_\beta$ on which $(\gamma_\sharp)_*$ acts nontrivially does not come of the image of $(\iota_\beta)_*$. Hence, in this case $\gamma$ could be trivial or nontrivial in homology.
\end{note}

By definition of higher Toda brackets (see \cite{W} for a modified version of Cohen's definition) it seems that a long Toda bracket of maps among spheres will reduce into a triple Toda bracket of three maps among complexes. Unfortunately, we do not have more detailed result in this direction, as in such cases the complexity of the complexes in the triple Toda bracket seems to increase which enforces to use more delicate and detailed computation such as those one appearing in the computation of $h(\kappa)$.

\subsection{Proof of Corollary \ref{Todaextension1}}
We break the proof into small lemmata. First, we deal with the case (i) of the corollary.

\begin{lmm}
Suppose $S\subset{_2\pi_{*>0}^s}$ such that for any $\alpha:S^i\to Q_0S^0$ belonging to $S$ either $\alpha_*=0$ or $e_*h(\alpha)=0$. Then, for any $\alpha\in S$, $\beta\in{_2\pi_j^s}$, and $\gamma:S^k\to Q_0S^0$ with $\gamma_*=0$, the Toda bracket $\{\alpha,\beta,\gamma\}$ maps trivially into $H_*(Q_0S^0;\Z/2)$.
\end{lmm}

\begin{proof}
If $\alpha_*=0$, then together with $\gamma_*=0$, applying Proposition \ref{tripleToda} proves the claim. So, suppose $\alpha_*\neq 0$. Then, by the assumption $e_*h(\alpha)=0$, for any $j,k>0$ the triple Toda bracket obtained from $S^{i+j+k}\to QS^{j+k}\to QS^k\to QS^0$ has the condition that $S^{i+j+k}\to QS^{j+k}$ is trivial in homology. So, with $\gamma_*=0$, applying Proposition \ref{tripleToda} proves the claim. The only possibility for $\{\alpha,\beta,\gamma\}$ mapping nontrivially into $H_*(QS^0;\Z/2)$ is that $j=k=0$, and we have
$$S^i\stackrel{\alpha}{\to} QS^0\stackrel{\beta}{\to} QS^0\stackrel{\gamma}{\to} QS^0$$
giving rise to a triple Toda bracket $\{\alpha,\beta,\gamma\}$. Note that for a nontrivial Toda bracket, none of the maps can be null. In this case, for the bracket to be defined, we need $\beta,\gamma\in{\pi_0^s}$ which are nontrivial, but their composition is. However, $\pi_0^s\simeq\Z$ has no zero divisors. Hence, this case will cannot arise. This completes the proof.
\end{proof}

Before proceeding further, we recollect some well-known facts which we shall use below (see \cite[Chapter 1 and Appenix 3]{Ravenel-Greenbook}, \cite[Chapter 16 and Chapter 18]{MosherTangora} for more details). First, for $\alpha:S^n\to QS^0$ then its $k$-th adjoint is $S^{n+k}\to QS^k$ which by Freudenthal's suspension theorem, provided $k>n$, it does factorise as a composite $S^{n+k}\to S^k\to QS^{n+k}$ implying that $e_*^{k}h(\alpha)=0$ for $k>n$. Second, we recall that for the Hopf invariant one elements we have
$$e_*h(\eta)=g_1^2\in H_2(QS^1;\Z/2), e_*^3h(\nu)=g_3^2\in H_6(QS^3;\Z/2), e_*^7h(\sigma)=g_7^2\in H_{14}(QS^7;\Z/2)$$
which, as quoted earlier, follows from \cite[Proposition 6.1.5]{Harper}(this also follows from \cite[Proposition 3.4]{Ecclesimmersion} by iterated application of homology suspension). Third, we recall ${_2\pi_i}$ for $i\leqslant 15$, ignoring the values of $i$ with trivial group, is determined by
$$
\begin{array}{c|ccccccccccccccccccccc}
i          & 0     && 1          && 2            && 3         && 6           &&7    &&8                 \\
\hline
{_2\pi_i^s}&\Z\{1\}&&\Z/2\{\eta\}&&\Z/2\{\eta^2\}&&\Z/8\{\nu\}&&\Z/2\{\nu^2\}&&\Z/16\{\sigma\} &&\Z/2\{\eta\sigma,\varepsilon\}
\end{array}$$
and
$$
\begin{array}{c|ccccccccccccccccccccccccccc}
i          &9                                &&10               &&11    && 14 &&15\\
\hline
{_2\pi_i^s}&\Z/2\{\nu^3,\eta^2\sigma,\mu_9\} &&\Z/2\{\eta\mu_9\}&&\Z/8\{\nu_1\}    &&\Z/2\{\sigma^2,\kappa\}&&\Z/2\{\eta\kappa\}\oplus\Z/{32}\{\sigma_1\}
\end{array}$$
where we have written $\nu_1$, $\sigma_1$, and $\eta_9$ for generators coming from ${_2\pi_*}J$. Finally, note that the degree $2$ map $S^0\to S^0$ or in general $2^r:S^0\to S^0$ does induce a translation in
$$H_*(QS;\Z/2)\simeq\Z/2[Q^I[1]*[-2^{l(I)}]:I\textrm{ admissible }]$$
sending $Q^I[1]*[-2^{l(I)}]$ to $Q^I[1]*[-2^{l(I)+r}]$. However, it induces multiplication by $2$ in $H_*(QS^n;\Z/2)$ for $n>0$ which is just the trivial in $\Z/2$-coefficients.\\
Finally, recall that at $p=2$ we have the well known James fibration $S^n\to \Omega\Sigma S^{n}\to\Omega\Sigma S^{2n}$ whose associated Serre exact sequence
$$\cdots\lra\pi_iS^n\stackrel{E}{\lra}\pi_{i+1}S^{n+1}\stackrel{H}{\lra}\pi_{i+1}S^{2n+1}\stackrel{P}{\lra}\pi_{i-1}S^n\lra\cdots$$
is known as the $EHP$-sequence.\\

Before proving part (ii) of the corollary, we proceed with some calculations which we will need later.

\begin{lmm}\label{todanusigma}
(i) The triple Toda brackets $\{\nu,2\nu,4\}$, $\{\nu,\eta^2,2\}$, and $\{\nu,8,\eta^2\}$, if defined, map trivially under the Hurewicz homomorphism $h:{_2\pi_*}QS^0\to H_*(QS^0;\Z/2)$.\\
(ii) The Toda brackets $\{\sigma,\nu^2,2\}$, $\{\sigma,2\nu,\nu\}$, $\{\sigma,\nu,2\nu\}$, and $\{\sigma,16,\nu^2\}$ map trivially under the Hurewicz homomorphism $h:{_2\pi_*}Q_0S^0\to H_*(Q_0S^0;\Z/2)$.
\end{lmm}

\begin{proof}
\textbf{Case of }$\{\nu,2\nu,4\}$. This bracket, if defined, represents an element in ${_2\pi_7^s}$. We show that the Hurewicz image of this element in $H_7(Q_0S^0;\Z/2)$ belongs to the kernel of homology suspension $e_*:H_*(Q_0S^0;\Z/2)\to H_{*+1}(QS^1;\Z/2)$. Hence, by Lemma \ref{kernelofsuspension} the image of $\{\nu,2\nu,4\}$ in $H_7(Q_0S^0;\Z/2)$, being a decomposable primitive as to be a square, but living in odd dimension. This will show that $h(\{\nu,2\nu,4\})=0$ in $H_7(Q_0S^0;\Z/2)$. To evaluate $e_*h(\{\nu,2\nu,4\})$ consider
$$\xymatrix{
S^7\ar[r]^-\nu & S^4\ar[r]^-{2\nu}\ar[d] & QS^1\ar[r]^-{4}\ar[d]         & QS^1\\
               & C_\nu\ar[d]^-p\ar[ru]   & C_{2\nu}\ar[ru]_-{4_\sharp}\\
               & S^8\ar[ru]_-{\nu^\flat}}$$
where $4_\sharp\circ\nu^\flat$ does represent $\{\nu,2\nu,4\}$ as an element of ${_2\pi_8}QS^1$. Since, $\nu$ and $4$ in the above diagram induce zero homomorphisms in homology, hence the composition $4_\sharp\circ\nu^\flat$ is trivial in homology by Proposition \ref{tripleToda}. Moreover, this does not depend on the choice of $4_\sharp$ and $\nu^\flat$, hence we have shown that the above Toda bracket, if represents a nontrivial element, has zero image in $H_*(QS^1;\Z/2)$. The proof then is complete by the above arguments.\\
\textbf{The other Cases.} The remaining cases of $\{\nu,\eta^2,2\}$, and $\{\nu,8,\eta^2\}$ as well as the cases of $\{\sigma,\nu^2,2\}$, $\{\sigma,2\nu,\nu\}$, $\{\sigma,\nu,2\nu\}$, and $\{\sigma,16,\nu^2\}$ fall into the same pattern. We, henceforth, deal with one case and leave the rest to the reader. Let's consider $\{\sigma,2\nu,\nu\}$ which we wish to evaluate its Hurewicz image in $H_*(Q_0S^0;\Z/2)$. To do this, we proceed with some unstable calculations. Note that if this bracket is defined, it then will determine an element of ${_2\pi_{14}^s}\simeq\Z/2\{\sigma^2\}$. By the Hopf invariant one result we know that $\sigma$ pulls back to $S^{15}\to S^8$ and $\nu,2\nu$ pull back to $S^7\to S^4$. Hence, we may think of $\sigma$ as a genuine map $S^{13}\to\Omega^2S^8$, $2\nu$ as a genuine map $S^6\to\Omega^3S^6$, and $\nu$ as $S^3\to\Omega\Omega^4S^4$. In order to compose these, we consider the extension of $2\nu$ into an $\Omega^2$-map $\Omega^2S^8\to \Omega^3S^6$, and the extension of $\nu$ into an $\Omega^3$-map $\Omega^3S^6\to\Omega^4S^4$. Now, for the Toda bracket $\{\sigma,2\nu,\nu\}$, we have
$$\xymatrix{
S^{13}\ar[r]^-\sigma & \Omega^2S^8\ar[r]^-{2\nu}\ar[d] & \Omega^3S^6\ar[d]\ar[r]^-\nu & \Omega^4S^4\\
                     & C_\sigma\ar[d]\ar[ru]           &  C_{2\nu}\ar[ru]_-{\nu_\sharp}\\
                     & S^{14}\ar[ru]_-{\sigma^\flat}}$$
with the composition $\nu_\sharp\circ\sigma^\flat$ representing the Toda bracekt $\{\sigma,2\nu,\nu\}$ as an element of ${_2\pi_{14}^s}\simeq\Z/2\{\sigma^2\}$. Now, if $h(\{\sigma,2\nu,\nu\})\neq 0$ then $\{\sigma,2\nu,\nu\}=\sigma^2$. However, this is a contradiction. To see this, note that $\sigma^2$ has the James-Hopf invariant $H(\sigma^2)=\sigma$, i.e. in the $EHP$ sequence
$$\cdots\lra\pi_{21}S^7\stackrel{E}{\lra}\pi_{22}S^{8}\stackrel{H}{\lra}\pi_{22}S^{15}\stackrel{P}{\lra}\pi_{20}S^{7}\lra\cdots$$
we have $H(\sigma^2)=\sigma$ which is nontrivial, hence it does not pull back further, even to $\pi_{21}S^{7}$, where as the above Toda bracket represents an element in ${_2\pi_{14}}\Omega^4S^4\simeq{_2\pi_{18}}S^4$. This contradiction, shows that $h(\{\sigma,2\nu,\nu\})=0$.
\end{proof}

\begin{rmk}\label{unstableH}
The unstable equalities $H(\eta^2)=\eta$, $H(\nu^2)=\nu$ and $H(\sigma^2)=\sigma$ are ought to be well known. But, they are also quite easy to deduce from homology computations. For instance, $H(\nu^2)$ is determined by the composite
$$S^{9}\stackrel{\nu}{\lra}S^6\stackrel{\nu}{\lra}\Omega S^4\stackrel{H}{\lra}\Omega S^7$$
where $H:\Omega\Sigma X\to \Omega \Sigma (X\wedge X)$ is the second James-Hopf map. Since $\nu$ has Hopf invariant one, hence the map $S^6\to\Omega\Omega S^4$ in homology sends $g_6\in H_6S^4$ to $g_3^2\in H_6\Omega S^4$ which then is mapped to $g_6\in H_6\Omega S^7$. An easy cohomology computation, using naturality of these operations, and that $\nu:S^9\to S^6$ is detected by $Sq^4$ on $g_6\in H_6S^6$ in its mapping cone, then shows that the above composition is detected by $Sq^4$ on $g_6\in H_6\Omega S^7$ in its mapping cone. Hence, it has to be $\nu$ or an odd multiple of it.
\end{rmk}

Now, we are ready to deal with the case (ii) for the corollary.

\begin{lmm}\label{bracketlmm2}
Suppose $S\subset{_2\pi_{*>0}^s}$ such that the Curtis conjecture holds on it. Then, the conjecture holds on the extensions of $\la S\ra$ obtained by triple Toda brackets $\{\alpha,\beta,\gamma\}$, with $\alpha\in S$, $\beta\in{_2\pi_*^s}$ is arbitrary, such that either one of the following conditions hold:\\
$(1)$ $\gamma_*=0$;\\
$(2)$ $\gamma_*\neq 0$ and $\alpha:S^{i+j+k}\to QS^{j+k}$ is nontrivial in homology.
\end{lmm}

\begin{proof}
By Theorem \ref{ideal} $\la S\ra$ is also a set on which the ideal holds. Hence, it is enough to prove that the lemma holds on $S$. Note that we know that Hopf invariant one and Kervaire invariant one elements maps nontrivially under $h:{_2\pi_*}Q_0S^0\to H_*(Q_0S^0;\Z/2)$. Hence, we have to show other cases map trivially.\\
For the case (1) if $\alpha_*=0$ then applying Proposition \ref{tripleToda} proves the claim. Moreover, if $\alpha:S^i\to QS^0$ is nonzero then by the assumption, $\alpha$ is either a Hopf invariant one element of or $\alpha=\theta_j$ is a Kerviare invariant one element. First, suppose $\alpha=\theta_j$ is a Kerviare invariant one element. Then, as quoted previously, by Madsen's result \cite[Theorem 7.3]{Madsenthesis} we have $e_*h(\alpha)=0$. This case then has been dealt with in the previous lemma. Hence, we only have to resolve the cases with $\alpha$ being a Hopf invariant one element. From now on, we collect the computations in a table which will resolve both the rest of case $(1)$ as well as case $(2)$. By the above, we start with
$$S^{i+j+k}\stackrel{\alpha}{\to} QS^{j+k}\stackrel{\beta}{\to} QS^k\stackrel{\gamma}{\to} QS^0$$
in order to form $\{\alpha,\beta,\gamma\}$.\\
\textbf{Case} $\alpha=\eta$. We have $i=1$ and by the above comments, in order for $S^{1+j+k}\to QS^{j+k}$ to be nontrivial in $\Z/2$-homology, we need $j+k\leqslant 1$. The following table collects the possible cases
$$
\begin{array}{cc|cc|cc|cc|cc|l}
j+k &\textrm{ }& j &\textrm{ }& k &\textrm{ }& \beta &\textrm{ }&\gamma &\textrm{ }& \{\alpha,\beta,\gamma\}\\
\hline
1   &\textrm{ }& 1 &\textrm{ }& 0 &\textrm{ }& \eta  &\textrm{ }&2      &\textrm{ }& \{\eta,\eta,2\}\textrm{ not defined as }\eta^2\neq 0\\
\hline
1   &\textrm{ }& 0 &\textrm{ }& 1 &\textrm{ }&   2   &\textrm{ }&\eta   &\textrm{ }& \{\eta,2,\eta\}=2\nu+\epsilon 4\nu\textrm{ for some }\epsilon\in\Z/2\\
\hline
0   &\textrm{ }& 0 &\textrm{ }& 0 &\textrm{ }&   2   &\textrm{ }&x   &\textrm{ }   & \textrm{there is no zero divisor $x$ in ${_2\pi_0^s}\simeq\Z$ with }2x=0\\
\end{array}$$
Note that $\{\eta,2,\eta\}=2\nu$ modulo $\eta^3=4\nu$ which allows to write $2\nu+\epsilon 4\nu$ for some $\epsilon\in\Z/2$. In any case, we may write $2\nu$ or $6\nu$ as $S^3\stackrel{2}{\to} S^3\stackrel{\nu}{\to} QS^0$ by commutativity of $\pi_*^s$ where the map on the left induces a multiplication by $2$. Hence, the only possible case in the above table maps trivially under $h:{_2\pi_3}QS^0\to H_3(QS^0;\Z/2)$. The case for $j+k=0$ is eliminated for the same reasons in the other cases. Hence, we shall not include them in the other two tables.\\
\textbf{Case }$\alpha=(2m+1)\nu$, that is $\alpha$ is an odd multiple of $\nu\in{_2\pi_3^s}$. Lets note that for $j+k=1$ then $\{\alpha,\beta,\gamma\}$, if defined, represents an element of ${_2\pi_5^s}\simeq 0$. Hence, we do not include this case in our table. We only consider the case $\alpha=\nu$ and the case of $\alpha$ being other odd multiples of $\nu$, at the prime $2$, are similar. We have the following
$$
\begin{array}{cc|ccc|ccc|clc|clc|cl}
j+k &\textrm{ }&& j &\textrm{ }&& k &\textrm{ }&& \beta &\textrm{ }&&\gamma &\textrm{ }&& \{\alpha,\beta,\gamma\}\\
\hline
3   &\textrm{ }&& 3 &\textrm{ }&& 0 &\textrm{ }&& 2\nu  &\textrm{ }&&4      &\textrm{ }&& \{\nu,2\nu,4\}\\
\hline
3   &\textrm{ }&& 2 &\textrm{ }&& 1 &\textrm{ }&& \eta^2&\textrm{ }&&\eta   &\textrm{ }&& \{\nu,\eta^2,\eta\}\textrm{ not defined as }\eta^3\neq 0\\
\hline
3   &\textrm{ }&& 1 &\textrm{ }&& 2 &\textrm{ }&&  \eta &\textrm{ }&&\eta^2 &\textrm{ }&& \{\nu,\eta,\eta^2\}\textrm{ not defined as }\eta^3\neq 0\\
\hline
3   &\textrm{ }&& 0 &\textrm{ }&& 3 &\textrm{ }&& 8     &\textrm{ }&&\nu    &\textrm{ }&& \{\nu,8,\nu\}=8\sigma\\
\hline
2   &\textrm{ }&& 2 &\textrm{ }&& 0 &\textrm{ }&& \eta^2&\textrm{ }&&2      &\textrm{ }&& \{\nu,\eta^2,2\} \\
\hline
2   &\textrm{ }&& 1 &\textrm{ }&& 1 &\textrm{ }&& \eta  &\textrm{ }&&\eta   &\textrm{ }&& \{\nu,\eta,\eta\}\textrm{ not defined as }\eta^2\neq 0\\
\hline
2   &\textrm{ }&& 0 &\textrm{ }&& 2 &\textrm{ }&&  8    &\textrm{ }&&\eta^2 &\textrm{ }&& \{\nu,8,\eta^2\} \\
\end{array}$$
Among the possible nontrivial brackets in the above table, $\{\nu,8,\nu\}=8\sigma$ maps trivially under $h:{_2\pi_7}QS^0\to H_7(QS^0;\Z/2)$ for reasons, similar to the case of $2\nu$. The remaining cases of $\{\nu,2\nu,4\}$, $\{\nu,\eta^2,2\}$, and $\{\nu,8,\eta^2\}$, if defined, map trivially into $H_*(QS^0;\Z/2)$ by Lemma \ref{todanusigma}.\\
\textbf{Case }$\alpha=(2m+1)\sigma$. For $j+k=4,5$, if $\{\alpha,\beta,\gamma\}$ is defined then it represents an element either in ${_2\pi_{12}^s}\simeq 0$ or ${_2\pi_{13}^s}\simeq 0$. We then have excluded these cases, as well as the case $j+k=0$. Moreover, for $j+k=7,6$ we have possible values such as $j=5,k=2$ so that we have to choose $\beta\in{_2\pi_5^s}\simeq 0$. Hence, in such cases, we cannot have a nontrivial Toda bracket $\{\alpha,\beta,\gamma\}$. We then have also excluded these cases. We have the following table
$$
\begin{array}{cc|ccc|ccc|clc|clc|cl}
j+k &\textrm{ }&& j &\textrm{ }&& k &\textrm{ }&& \beta   &\textrm{ }&&\gamma &\textrm{ }&& \{\alpha,\beta,\gamma\}\\
\hline
7   &\textrm{ }&& 7 &\textrm{ }&& 0 &\textrm{ }&& 2\sigma &\textrm{ }&&8      &\textrm{ }&& \{\sigma,2\sigma,8\}\\
\hline
7   &\textrm{ }&& 6 &\textrm{ }&& 1 &\textrm{ }&& \nu^2   &\textrm{ }&&\eta   &\textrm{ }&& \{\sigma,\nu^2,\eta\}\\
\hline
7   &\textrm{ }&& 1 &\textrm{ }&& 6 &\textrm{ }&& \eta    &\textrm{ }&&\nu^2  &\textrm{ }&& \{\sigma,\eta,\nu^2\}\textrm{ not defined as }\eta\sigma\neq 0\\
\hline
7   &\textrm{ }&& 0 &\textrm{ }&& 7 &\textrm{ }&& 8       &\textrm{ }&&\sigma &\textrm{ }&& \{\sigma,16,\sigma\}\\
\hline
6   &\textrm{ }&& 6 &\textrm{ }&& 0 &\textrm{ }&& \nu^2   &\textrm{ }&&2      &\textrm{ }&& \{\sigma,\nu^2,2\}\\
\hline
6   &\textrm{ }&& 3 &\textrm{ }&& 3 &\textrm{ }&& 2\nu    &\textrm{ }&&\nu    &\textrm{ }&& \{\sigma,2\nu,\nu\}\\
\hline
6   &\textrm{ }&& 3 &\textrm{ }&& 3 &\textrm{ }&&  \nu    &\textrm{ }&&2\nu   &\textrm{ }&& \{\sigma,\nu,2\nu\}\\
\hline
6   &\textrm{ }&& 0 &\textrm{ }&& 6 &\textrm{ }&& 16      &\textrm{ }&&\nu^2  &\textrm{ }&& \{\sigma,16,\nu^2\}\\
\hline
3   &\textrm{ }&& 3 &\textrm{ }&& 0 &\textrm{ }&& \nu     &\textrm{ }&&8      &\textrm{ }&& \{\sigma,\nu,8\}\\
\hline
3   &\textrm{ }&& 2 &\textrm{ }&& 1 &\textrm{ }&& \eta^2&\textrm{ }&&\eta     &\textrm{ }&& \{\sigma,\eta^2,\eta\}\textrm{ not defined as }\eta^2\sigma\neq 0\\
\hline
3   &\textrm{ }&& 1 &\textrm{ }&& 2 &\textrm{ }&&  \eta &\textrm{ }&&\eta^2   &\textrm{ }&& \{\sigma,\eta,\eta^2\}\textrm{ not defined as }\eta\sigma\neq 0\\
\hline
3   &\textrm{ }&& 0 &\textrm{ }&& 3 &\textrm{ }&& 16    &\textrm{ }&&\nu      &\textrm{ }&& \{\sigma,16,\nu\}\\
\hline
2   &\textrm{ }&& 2 &\textrm{ }&& 0 &\textrm{ }&& \eta^2&\textrm{ }&&2        &\textrm{ }&& \{\sigma,\eta^2,2\}\textrm{ not defined as }\eta^2\sigma\neq 0 \\
\hline
2   &\textrm{ }&& 1 &\textrm{ }&& 1 &\textrm{ }&& \eta  &\textrm{ }&&\eta     &\textrm{ }&& \{\sigma,\eta,\eta\}\textrm{ not defined as }\eta\sigma\neq 0\\
\hline
2   &\textrm{ }&& 0 &\textrm{ }&& 2 &\textrm{ }&&  16   &\textrm{ }&&\eta^2   &\textrm{ }&& \{\sigma,16,\eta^2\}. \\
\end{array}$$
The brackets $\{\sigma,2\sigma,8\}$, $\{\sigma,\nu^2,\eta\}$, and $\{\sigma,16,\sigma\}$, if defined, represent an element in ${_2\pi_{15}^s}\simeq\Z/2\{\eta\kappa\}\oplus\Z/32\{\sigma_1\}$. The element $\eta\kappa$ maps trivially under the Hurewicz homomorphism by Theorem \ref{main0}. The element $\sigma_1$, coming from ${_2\pi_*J}$, also maps trivially into $H_*(QS^0;\Z/2)$ by \cite[Theorem 1]{Za}. The brackets $\{\sigma,\nu,8\}$, $\{\sigma,16,\nu\}$, if defined, represent an element of ${_2\pi_{11}^s}\simeq\Z/8\{\nu_1\}$ with $\nu_1$ being a generator coming from ${_2\pi_*}J$ which we know maps trivially under $h:{_2\pi_*}Q_0S^0\to H_*(Q_0S^0;\Z/2)$ by \cite[Theorem 1]{Za}. For $\{\sigma,16,\eta^2\}$, if it is defined, it will represent an element of ${_2\pi_{10}^s}\simeq \Z/2\{\eta\mu_9\}$ where by Theorem \ref{main0} we have $h(\eta\mu_9)=0$. Lemma \ref{todanusigma} then will take care of the remaining cases of $\{\sigma,\nu^2,2\}$, $\{\sigma,2\nu,\nu\}$, $\{\sigma,\nu,2\nu\}$, and $\{\sigma,16,\nu^2\}$. This completes the proof.
\end{proof}

\section{Hurewicz homomorphism and the $EHP$-sequence}
The only observation of this section, is a mere generalisation of the technique used Lemma \ref{todanusigma}. The result may look a little more than a triviality, but we have already applied it to eliminate some Toda brackets from giving rise to spherical classes in $H_*(Q_0S^0;\Z/2)$. By Theorem \ref{main0} our main focus is on the elements of ${_2\pi_*^s}$ which are not decomposable, and represented by derived products, of which in this section we consider triple Toda brackets. For $f\in{_2\pi_n^s}$ we know that it pulls back to an element in ${_2\pi_{2n+1}}S^{n+1}$. As we work at the prime $p=2$, then we may apply the $EHP$-sequence in an iterative manner to see how far $f$ does pull back. We define $h_f\in\mathbb{N}$ to be the least positive integer such that $f$ pulls back to $f\in{_2\pi_{n+1+h_f}}S^{1+h_f}$, and it does not pull back to ${_2\pi_{n+h_f}}S^{h_f}$, that is $f$ maps nontrivially under the James-Hopf invariant $H$, which is different from the stable Hopf invariant as noted in Remark \ref{unstableH}, in the $EHP$-sequence
$$\cdots\to{_2\pi_{n+h_f}}S^{h_f}\stackrel{E}{\to}{_2\pi_{n+1+h_f}}S^{1+h_f}\stackrel{H}{\to}{_2\pi_{n+1+h_f}}S^{2h_f+1}\stackrel{P}{\to}
{_2\pi_{n+h_f-1}}S^{h_f}\to\cdots$$
i.e. $H(f)\neq 0$. For instance, we have
$$h_\eta=h_{\eta^2}=1,\ h_\nu=h_{\nu^2}=3,\ h_\sigma=h_{\sigma^2}=7.$$
As another example, take $\eta_3\in{_2\pi_{2^i}^s}$ which we know both $\sigma\eta\in{_2\pi_{16}}S^8$ and $\eta\sigma\in{_2\pi_{15}}S^7$ map to it under the stablisation map which implies that $h_{\eta_3}=6$. Note that, by Freudenthal suspension theorem, for $f\in{_2\pi_i^s}$ we have $h_f\leqslant i$. Using this notation, we state our observation as follows.
\begin{prp}\label{EHP}
Suppose $\alpha\in{_2\pi_i^s}$, $\beta\in{_2\pi_j^s}$ and $\gamma\in{_2\pi_k^s}$ are given, so that the triple Toda bracket $\{\alpha,\beta,\gamma\}$ is defined and its image under the Hurewicz homomorphism ${_2\pi_{i+j+k+1}}Q_0S^0\lra H_{i+j+k+1}(Q_0S^0;\Z/2)$ is nontrivial. Then
$$h_{\{\alpha,\beta,\gamma\}}\leqslant\max\{h_\alpha,h_\beta,h_\gamma\}.$$
\end{prp}

Such an observation becomes useful, once we have some information on ${_2\pi_{i+j+k+1}^s}$ and the unstable behaviour of its generators in the $EHP$-sequences, or $EHP$-spectral sequence. For instance, in the proof of Lemma \ref{todanusigma} the fact that if $h(\{\sigma,\nu,2\nu\})\neq 0$ we deduced that $\{\sigma,\nu,2\nu\}$ has to be the same as $\sigma^2$. The latter implies that $h_{\{\sigma,\nu,2\nu\}}=3$ whereas $h_{\sigma^2}=7$. This then allowed us to eliminate the possibility of $h(\{\sigma,\nu,2\nu\})\neq 0$.
\begin{proof}
We consider one of the cases, and the other cases are similar. Suppose $h_\alpha\leqslant h_\gamma\leqslant h_\beta$. We think of $\alpha$ as $S^{i+j+k}\to\Omega^{1+h_\alpha}\Sigma^{1+h_\alpha}S^{j+k}$, $\beta$ as $S^{j+k}\to\Omega^{1+h_\beta}\Sigma^{1+h_\beta}S^k$, and $\gamma$ as $S^k\to\Omega^{1+h_\gamma}\Sigma^{1+h_\gamma}S^0$. In order to form an unstable Toda bracket, we have to be able to compose these maps. To compose $\gamma$ and $\beta$, we think of $\gamma$ as the extension of the composite
$$S^k\to\Omega^{1+h_\gamma}\Sigma^{1+h_\gamma}S^0\stackrel{E^{h_\beta-h_\gamma}}{\lra}\Omega^{1+h_\beta}\Sigma^{1+h_\beta}S^0$$
to an $\Omega^{1+h_\beta}$-map $\Omega^{1+h_\beta}\Sigma^{1+h_\beta}S^k\to\Omega^{1+h_\beta}\Sigma^{1+h_\beta}S^0$ where $E^{h_\beta-h_\gamma}$ is the iterated suspension map. Similarly, we think of $\beta$ as its extension to an $\Omega^{1+h_\beta}$-map $S^{j+k}\to\Omega^{1+h_\beta}\Sigma^{1+h_\beta}S^k$. Finally, we consider $\alpha$ as the composite
$S^{i+j+k}\to\Omega^{1+h_\alpha}\Sigma^{1+h_\alpha}S^{j+k}\stackrel{E^{h_\beta-h_\alpha}}{\lra}\Omega^{1+h_\beta}\Sigma^{1+h_\beta}S^{j+k}$. We then may form the diagram
$$\xymatrix{
S^{i+j+k}\ar[r]^-\alpha &\Omega^{1+h_\beta}\Sigma^{1+h_\beta}S^{j+k}\ar[r]^-\beta\ar[d]
&\Omega^{1+h_\beta}\Sigma^{1+h_\beta}S^{k}\ar[r]^-\gamma\ar[d]&\Omega^{1+h_\beta}\Sigma^{1+h_\beta}S^0\\
                        & C_\alpha\ar[ru]\ar[d]                                        & C_\beta\ar[ru]_-{\gamma_\sharp}\\
                        & S^{i+j+k+1}\ar[ru]_-{\alpha^\flat}
                        }$$
with the composition $\gamma_\sharp\circ\alpha^\flat$ representing $\{\alpha,\beta,\gamma\}$ as an element of the unstable group ${_2\pi_{i+j+k+1}}\Omega^{1+h_\beta}\Sigma^{1+h_\beta}S^0$, up to the indeterminacy in choosing the extension and co-extension maps. The homology of the composite $\gamma_\sharp\circ\alpha^\flat$ for any of such choices, determines $h(\{\alpha,\beta,\gamma\})$. The assumption that $h(\{\alpha,\beta,\gamma\})\neq 0$ implies that this Toda bracket represents an element of ${_2\pi_{i+j+k+h_\beta+1}}S^{h_\beta+1}$ as a pull back of the element of ${_2\pi_{i+j+k+1}^s}$ represented by the stable Toda bracket. This then tells us that the unstable element $\{\alpha,\beta,\gamma\}$ at least is born on $S^{1+h_\beta}$ which may pull back or no. In any case, this shows that
$$h_{\{\alpha,\beta,\gamma\}}\leqslant h_\beta=\max\{h_\alpha,h_\beta,h_\gamma\}.$$
\end{proof}

\section{Hurewicz homomorphism and homotopy operations}\label{operationsection}
The motivation for this section is again provided by Theorem \ref{main0} and Theorem \ref{ideal}. Given $S\subset{\pi_{*>0}^s}$, we consider extensions of $S$ by applying homotopy operations. The operations that we consider here, are of specific type; writing $D_r$ for the $r$-adic construction $(E\Sigma_r)\ltimes_{\Sigma_r}(-)^{\wedge r}$, then a given element $\alpha\in\pi_mD_rS^n$ determines an operation $\alpha^*:\pi_n^s\to\pi_m^s$ which sends $f\in\pi_n^s$ to the element given by the composite
$$S^m\lra D_rS^n\stackrel{D_rf}{\lra}D_rS^0={B\Sigma_r}_+\lra S^0$$
where $D_rS^0\lra S^0$ is induced by the multiplication of $S^0$ as an $H_\infty$ spectrum. Of course, it is not guaranteed that such operations always will exist or give rise to new elements. But, when they are defined then we can ask about their Hurewicz image. First, we record the following.

\begin{lmm}
Let $f:X\to Y$ be a stable map, i.e. a map of suspension spectra, between path connected complexes such that $\Omega^\infty f:QX\to QY$ is trivial in homology. Then $\Omega^\infty D_rf:QD_rX\to QD_rY$ is trivial in homology.
\end{lmm}

\begin{proof}
Since $X$ and $Y$ are path connected, then by \cite[Theorem 2.7]{Kuhn-suspension} we have a factorisation of $D_rf$ as
$$D_rf:\Sigma^\infty D_rX\to\Sigma^\infty QX\stackrel{\Sigma^\infty\Omega^\infty f}{\longrightarrow}\Sigma^\infty QY\to \Sigma^\infty D_rY.$$
Hence, upon applying $\Omega^\infty$ we have a factorisation for $\Omega^\infty (D_rf)$ as
$$\Omega^\infty(D_rf):QD_rX\to QQX\stackrel{Q(\Omega^\infty f)}{\longrightarrow} QQY\to \Sigma^\infty QD_rY$$
Now, the map at the middle, $Q(\Omega^\infty f)$, is induced by applying $Q$ to the map of spaces $\Omega^\infty f: QX\to QY$, then $(\Omega^\infty f)_*=0$ implies that $Q(\Omega^\infty f)_*=0$ and consequently $(\Omega^\infty D_rf)_*=0$. This completes the proof.
\end{proof}

The following then is more or less expected.

\begin{prp}
Let $\alpha\in\pi_mD_rS^n$ with $n>0$. Suppose $f\in{_2\pi_n^s}$ such that it maps trivially under $h:{_2\pi_*}Q_0S^0\to H_*(Q_0S^0;\Z/2)$. Then $h(\alpha^*(f))=0$.
\end{prp}

\begin{proof}
In order to compute $h(\alpha^*(f))\in H_*(Q_0S^0;\Z/2)$ we need to compute the homology of the stable adjoint of $S^m\lra D_rS^n\stackrel{D_rf}{\lra}D_rS^0=Q({B\Sigma_r}_+)\lra S^0$ given by
$$S^m\lra QD_rS^n\stackrel{\Omega^\infty D_rf}{\lra}QD_rS^0=Q({B\Sigma_r}_+)\lra QS^0$$
noting that $f:S^n\to S^0$ is a map of suspension spectra with $S^n$ being path connected. By Kahn Priddy theorem, we may factorise $f$ as a composite of stable maps $S^n\stackrel{\overline{f}}{\to} P\stackrel{\lambda}{\to}S^0$. For a moment, writing $f^s:S^n\to Q_0S^0$ for the stable adjoint of $f$, a choice for $\overline{f}$, up to an automorphism of ${_2\pi_n^s}$, is obtained by considering $t\circ f^s:S^n\to QP$ and then taking its stable adjoint which yields a stable map $S^n\to P$. Now, by Kahn-Priddy theorem, for such a choice of $\overline{f}$, we have $\lambda\overline{f}=f$, up to an automorphism of ${_2\pi_n^s}$, and that $(\Omega^\infty\overline{f})_*=0$. Now, for computing $h\alpha^*(f)$ we consider the composition
$$S^m\lra QD_rS^n\stackrel{\Omega^\infty D_r\overline{f}}{\lra}QD_rP\stackrel{\Omega^\infty D_r\lambda}{\lra}QD_rS^0=Q({B\Sigma_r}_+)\lra QS^0.$$
Since $\overline{f}:S^n\to P$ is stable map between connected complexes, then by the above lemma, the equality $(\Omega^\infty\overline{f})_*=0$ implies that $(\Omega^\infty D_r\overline{f})_*=0$, and consequently the above composition is trivial in homology. Since the composition $t\lambda$ is a $2$-local equivalence, then it induces an isomorphism in $\Z/2$-homology. So, the result in the above computations, remains unchanged after multiplying by an automorphism of ${_2\pi_n^s}$. This then completes the proof.
\end{proof}


As an application, let $\tau_i$ denote Bruner's family which is constructed using homotopy operation $\cup_1$ as $\tau_i:=\cup_1(\eta_i)-\eta\eta_{i+1}$ \cite{Bruner}. We have the following.

\begin{lmm}
The $\tau_i$ elements of Bruner family, map trivially under the Hurewicz homomorphism ${_2\pi_{2^{i+1}+1}}Q_0S^0\to H_{2^{i+1}+1}Q_0S^0$.
\end{lmm}

\begin{proof}
By the above proposition, since $h(\eta_i)=0$ then $h(\cup_1(\eta_i))=0$. Moreover, since $\eta$ and $\eta_{i+1}$ are not in the same grading, then by Theorem \ref{main0} $h(\eta\eta_{i+1})=0$. Hence, $h(\tau_i)=0$.
\end{proof}

\section{Hurewicz homomorphism and Adams filtration}
This section is devoted to obtain a numerical conditions on the elements of ${_2\pi_*^s}$ which map nontrivially. Let $D_k$ be the Dickson algebra on $k$ variables and $A$ be the mod $2$ Steenrod algebra. Recall the Lannes-Zarati homomorphism
$$\varphi_k:\ext_A^{k,k+i}(\Z/2,\Z/2)\lra(\Z/2\otimes_A D_k)_i^*$$
where its domain is the $E_2$-term of the Adams spectral sequence. By \cite[Corollary 3.3, Proposition 3.5]{Madsen}, and decomposition of Steenrod sqaures into composition of $Sq^{2^t}$, the target of $\varphi_k$ is the $R$-submodule of $H_*(Q_0S^0;\Z/2)$ generated by elements of length $k$, (see \cite{HungPeterson} for more detailed discussion as well as \cite[Theorem 1.6]{Ku}). Moreover, we know by Kuhn's work \cite[Corollary 1.7]{Ku} that such a statement holds after replacing $S^0$ with $\Sigma^\infty X$ for any space $X$, and that if $X$ itself is a suspension, then we may consider a map, analogue to $\varphi_k$, which at the prime $p=2$ looks like
$$\ext_A^{s,t}(H^*X,\Z/2)\to \mathcal{R}_sH_*X\subset H_*(D_{2^s}X)$$
which we continue to denote by $\varphi_k$, and is induced by the Hurewicz homomorphism ${_2\pi_*}QX\to H_*QX$. We shall use
$$H_*(Q_0S^0;\Z/2)\simeq\Z/2[Q^I[1]*[-2^{l(I)}]:I\textrm{ admissible}]$$
so the image of a cycle in $\ext_A^{k,k+i}(\Z/2,\Z/2)$ under $\varphi_k$, modulo decomposable terms, will be a linear combination of classes $Q^I[1]*[-2^{l(I)}]$ with $l(I)=k$. Here, the action of the homology suspension on the generators is determined by
$$e_*(Q^I[1]*[-2^{l(I)}])=Q^Ig_1$$
where $g_1\in H_1S^1$ is a generator. We then have the following.
\begin{lmm}
Suppose $f\in{_2\pi_n}Q_0S^0$ of Adams filtration $k$, i.e. $k$ is the least positive integer where $f$ is represented by a cycle in $\ext^{k,k+n}_A(\Z/2,\Z/2)$, and $h(f)\neq 0$ where $h:{_2\pi_n}Q_0S^0\to H_n(Q_0S^0;\Z/2)$ is the Hurewicz homomorphism. The following statement then hold.\\
(i) If $e_*h(f)\neq 0$ then $n\geqslant 2^k-1$.\\
(ii) If $e_*h(f)=0$, $h(f)=\xi^{2^t}$ with $e_*\xi\neq 0$, and $h(f)\in\im(\varphi_k)$ then $n\geqslant 2^{k}-2^t$.
\end{lmm}

\begin{proof}
(i) Since $h(f)\neq 0$ with $e_*h(f)\neq 0$ then we may write
$$h(f)=\int_{\epsilon_I\in\Z/2}\epsilon_I Q^I[1]*[-2^{l(I)}]+\textrm{ decomposable terms}$$
where the sum runs over admissible sequences $I$ such that there exists at least one $I$ with $\epsilon_I=1$. Consider the adjoint of $f:S^n\to Q_0S^0$ as $f_1:S^{n+1}\to QS^1$. Together with Kuhn's result quoted above, since $h(f_1)\in\im(\varphi_k)$, we may write
$$e_*h(f)=\int_{l(I)\geqslant k,\epsilon_I\in\Z/2}\epsilon_I Q^Ig_1.$$
Let $l(f):=\max\{l(I):\epsilon_I=1\}$, and note that $l(f)\geqslant k$. Write $j_{2^{l(f)}}$ for the $2^{l(f)}$-th stable James-Hopf invariant $QS^1\to QD_{2^{l(f)}}S^1$; it acts like projection on the elements of height $2^{l(f)}$ such as $Q^Ig_1$ with $l(I)=l(f)$, and all term $Q^Jg_1$ with $l(J)<l(f)$ are killed under this map. Hence,
$$(j_{2^{l(f)}})_*h(f)=\int_{l(I)=l(f),\epsilon_I\in\Z/2}\epsilon_I \underline{Q^Ig_1}\neq 0.$$
Here, $\underline{Q^Ig_1}$ denotes the image of $Q^Ig_1$ in $H_*D_{2^{l(f)}}S^1$. The space $D_{2^{l(f)}}S^1$ has its bottom cell in dimension $2^{l(f)}$, and the homology of $H_*QD_{2^{l(f)}}S^1$ vanishes in dimensions below $2^{l(f)}$. Hence, a necessary condition for $h(f)\neq 0$ is that $n+1\geqslant 2^{l(f)}\geqslant 2^k$. This proves the desired inequality.\\
(ii) By Proposition \ref{kernelofsuspension}, $e_*h(f)=0$ implies that $h(f)$ is a decomposable. Hence, being a primitive decomposable, $h(f)$ is a square term. Since, $h(f)\in\im(\varphi_k)$ hence
$$h(f)=\int_{\epsilon_I\in\Z/2,l(I)\geqslant k,\ex(I)=0}\epsilon_I Q^I[1]*[-2^{l(I)}]+ D$$
where $D$ is a sum of terms each of which is a finite product of elements of the form $Q^J[1]*[-2^{l(J)}]$ with $l(J)\geqslant k$. Since, $h(f)$ is a square, we may assume that $h(f)=\xi^{2^t}$ for some $t>0$ where
$$\xi=\int_{\ex(\dot{I})>0} \epsilon_IQ^{\dot{I}}[1]*[-2^{l(\dot{I})}]+D'$$
where $D'$ is a sum of decomposable terms. This means that $e_*\xi\neq 0$. Now, we may apply the same technique as in part (i) to show that $\dim\xi+1\geqslant 2^{l(f)-t}>2^{k-t}$. Hence, noting that $\dim h(f)=2^t(\dim\xi)$, $h(f)\neq 0$ implies that
$$\dim\xi\geqslant 2^{k-t}-1\Rightarrow \dim h(f)\geqslant 2^t(2^{k-t}-1)=2^k-2^t.$$
This completes the proof.

\end{proof}




\section{Final comments}
The aim of this note was to collect some relatively quick observations on the type of spherical classes in $H_*(Q_0S^0;\Z/2)$ that one could derive from using the internal structure of the stable homotopy ring as well as the unstable homotopy groups of spheres. We have not considered using spectral sequence arguments. As pointed out by Peter May, Wellington's Chicago thesis \cite{Wellington}, is a very rich source for such computations. We believe that our computations on the $EHP$-sequence sheds light on the possible applications of the $EHP$-spectral sequence arguments to the problem. We liked to see if there is any result on generating the stable homotopy ring by using homotopy operations, as suggested to the author by Gerald Gaudens. But we were not able to find any structural result in this direction. We hope to come back to some of these points later on.\\

\textbf{Acknowledgements.}
I am grateful to Peter May for his comments on the earlier version of this paper, to Nick Kuhn for many helpful comments and discussions and for pointing out gaps in some arguments of the earlier versions of this work. Thanks also to Gerald Gaudens and Nguy\^{e}n H.V. Hu'ng for discussions on the conjecture, while I was visiting LAERMA at Universite de Angers during October 2014. I am grateful to Gerald Gaudens and Geoffrey Powell for an invitation to LAERMA, and for their hospitality. I acknowledge the partial support from IPM, University of Tehran, and LAERMA which made this visit possible. Last, but not least, I am very grateful to Peter Eccles whom I have learnt from him very much on the topic when I was his student at Manchester; the present paper is an elaboration on what I have learnt from him.

\end{document}